\documentclass[11pt]{article}

\usepackage{amsmath,amsthm,amsfonts,amssymb,amscd, amsxtra, mathrsfs,textcomp,verbatim,inputenc}
\usepackage{url}
\usepackage{xcolor}
\usepackage{mathtools}
\usepackage{booktabs}
\usepackage[shortlabels]{enumitem}

\usepackage{hyperref}
\usepackage{soul} 

\usepackage{algorithm}

\theoremstyle{plain}
\theoremstyle{definition}
\newtheorem{theorem}{Theorem}
\newtheorem{lemma}[theorem]{Lemma}
\newtheorem{definition}[theorem]{Definition}
\newtheorem{corollary}[theorem]{Corollary}
\newtheorem{proposition}[theorem]{Proposition}

\theoremstyle{remark}
\newtheorem{remark}{Remark}
\newtheorem{example}{Example}

\begin{document}

\title{A Busemann hybrid projection-proximal point algorithm for optimization problems on Hadamard manifolds}
\author{R. D\'iaz Mill\'an\thanks{School of Information Technology, Deakin University, Geelong, Australia (Email: {\tt r.diazmillan@deakin.edu.au}, {\tt julien.ugon@deakin.edu.au})}
\and O.P. Ferreira\thanks{Institute of Mathematics and Statistics, Federal University of Goias, Avenida Esperan\c{c}a, s/n, Campus II, Goi\^ania, GO - 75690-900, Brazil (Email: {\tt orizon@ufg.br}, {\tt mauriciolouzeiro@ufg.br})}
\and M.S. Louzeiro\footnotemark[2]
\and J.Ugon\footnotemark[1]
}
\date{}

\maketitle

\begin{abstract}
We study optimization problems on Hadamard manifolds, motivated by recent advances in geometric approaches to optimization on curved spaces, particularly those involving the structure of Busemann functions. We introduce a projection based variant of the proximal point algorithm, termed the \emph{Busemann hybrid projection proximal point algorithm}, which replaces Euclidean hyperplanes with horospheres defined via convex Busemann functions. The algorithm performs projections in closed form using the gradients of these functions, resulting in a geometrically intrinsic scheme that requires no tangent space linear solves. We allow for inexact subgradient evaluations and prove global convergence under controlled inexactness, with a relative error level strictly below one. We establish a Fejér type descent and sublinear complexity with a rate proportional to the inverse square root of the iteration count, and show that the exact variant coincides with the classical Riemannian proximal point algorithm. The framework clarifies the role of Busemann based subdifferentials in optimization on spaces of nonpositive curvature.

\noindent \textbf{keywords:}{Hadamard manifolds, Busemann function,  Busemann subdifferential, horosphere projection, proximal point algorithm}

\end{abstract}

\section{Introduction}

In this paper, we study the following problem, which generalizes classical convex minimization to the setting of Riemannian geometry:
\begin{equation} \label{eq:problem}
\min_{p \in {\mathbb{M}}} f(p),
\end{equation}
where \( {\mathbb{M}} \) is a Hadamard manifold and \( f \colon {\mathbb{M}} \to {\mathbb{R}} \) is a proper, lower semi-continuous (lsc), and convex function. Hadamard manifolds, complete, simply connected Riemannian manifolds with nonpositive sectional curvature, provide a natural framework for optimization beyond the Euclidean setting. These manifolds  exhibit several desirable geometric properties,  the exponential map at any point is a global diffeomorphism, geodesics are uniquely defined between any pair of points, and the squared distance function is convex along geodesics. As a result, many tools from convex analysis and variational optimization extend meaningfully to this setting, enabling the design of efficient algorithms that exploit the intrinsic curvature of the manifold. Convex functions on Hadamard manifolds are defined in terms of the convexity of their restriction to geodesics, and the notion of subdifferential and proximal maps can be generalized via the Riemannian exponential and logarithmic maps. This geometric framework arises naturally in various applications involving structured data with non-Euclidean constraints, such as symmetric positive definite matrices  and hyperbolic embeddings. Consequently, designing algorithms for problem~\eqref{eq:problem} that respect the underlying curvature and exploit manifold-specific structures is both theoretically appealing and practically relevant.

The proximal point algorithm (PPA), originally introduced in the ambient linear space by Martinet~\cite{Martinet1970} and further developed by Rockafellar~\cite{Rockafellar76}, plays a central role in optimization and monotone operator theory, due to its strong convergence properties and broad applicability. A refinement of the classical algorithm was proposed by Solodov and Svaiter~\cite{SolodovBenar1999}, who introduced a projection step onto a hyperplane derived from an inexact proximal iteration, improving the control over inexactness while maintaining convergence. The extension of the PPA to nonlinear  setting  was initiated in~\cite{Ferreira2002}, where the algorithm was adapted to solve convex optimization problems on Hadamard manifolds.  This approach was further generalized in~\cite{LiLopesMartin-Marquez2009} to handle the problem of finding singularities of monotone vector fields, introducing tools that contributed to the development of Riemannian monotone operator theory. Since then, various works have investigated the PPA and its variants in the Riemannian context; see, for instance,~\cite{BatistaBentoFerreira2016,BentoCruzNeto2014,BentoFerreiraOliveira2015,BentoNetoOliveira2016,Kajimura2019,Nicolae2018,LiYao2012,WangLi2016,LiMordukhovichWang2011,Wang2010,WangLopez2019}. Beyond Riemannian geometry, the PPA has also been studied in the more general setting of geodesically complete metric spaces of non-positive curvature, known as \( CAT(0) \) spaces, as proposed by Bačák~\cite{Bacak2013}. This extension made it possible to apply proximal-type algorithms in broader nonsmooth and non-linear  settings, and motivated further developments; see~\cite{Bacak2014,Chaipunya2017,Cholamjiak2015,Cuntavepanit2018,Lerkchaiyaphum2018,Pakkaranang2018,Phuengrattana2018,Ugwunnadi2018}.

Recent advances in optimization in Riemannian setting  have highlighted the importance of developing tools adapted to the curvature and topology of the underlying space. In this context, \emph{Busemann functions}, which capture asymptotic behavior of geodesics, emerged as powerful instruments for redefining concepts of convexity, subgradients, and projection-like operations on Hadamard spaces. These functions play a role analogous to affine functions in Euclidean space, inducing analogues of hyperplanes (horospheres) and halfspaces (horoballs), thereby enabling a wide range of practical applications; see, for example, \cite{fan2023, suganthan2025} and the references therein. Several recent works have explored this perspective. The paper~\cite{Goodwin2024} introduces a notion of subdifferentiability based on Busemann functions and proposes subgradient-type algorithms for optimization on Hadamard spaces. A parallel line of work~\cite{Criscitiello2025} develops the concept of \emph{horospherical convexity}, also grounded in Busemann geometry, as a means to overcome certain limitations of classical geodesic convexity. In~\cite{Lewis2024} is presented a horospherically-based subgradient algorithm that avoids tangential constructions and lower curvature bounds, applying even to spaces with singularities such as CAT(0) cubical complexes.

Building on these geometric insights, we propose a Riemannian analogue of the hybrid projection based proximal algorithm of~\cite{SolodovBenar1999}, adapted to the setting of Hadamard manifolds. Our method, termed the \emph{Busemann hybrid projection proximal point algorithm}(BHPPM), exploits the convexity and asymptotic structure of Busemann functions to define projection steps onto horospheres, which act as nonlinear analogues of Euclidean hyperplanes. These horospheres serve as separating sets, and the projection at each iteration is computed in closed form using the gradient of the  Busemann function~\cite{Ballmann1985,BridsonHaefliger1999,Busemann1955}. This strategy avoids tangent space linear solves and curvature dependent estimates, yielding a conceptually simple and geometrically intrinsic scheme. The algorithm allows inexact subgradient evaluations controlled by a relative error level strictly below one and reduces to the exact Riemannian proximal point algorithm when this level vanishes. The projection step enforces a Fejér type geometric descent and is essential for convergence in the presence of inexactness. Moreover, we establish a sublinear complexity bound for the inexact method, showing that the decrease in an appropriate merit function is on the order of the inverse square root of the iteration count. Taken together, these results also yield an alternative convergence proof for the exact case and point to new directions for structured optimization techniques on Hadamard manifolds.

The remainder of the paper is organized as follows. Subsection~\ref{sec:ContReLiterature} states our contributions and situates them within the recent literature. Section~\ref{sec:int.r} reviews essential concepts, notation, and preliminary results on Hadamard manifolds and convexity. Section~\ref{sec:bfNewDef} revisits Busemann functions on Hadamard manifolds, fixes notation, and collects key properties with illustrative examples. Section~\ref{sec:NewDefSub} introduces a notion of subgradient adapted to the geometry induced by Busemann functions and establishes basic properties, including inclusion in the classical subdifferential, a simple chain rule for nondecreasing convex scalars, and max type constructions under an alignment condition, accompanied by examples. Section~\ref{eq:proj} presents the geometric foundations of the projection step in the Busemann proximal point algorithm, emphasizing the role of horospheres and their closed form projections. Section~\ref{sec:proxLine} introduces the Busemann proximal point algorithm, including its inexact variant, and proves key descent properties, global convergence, and sublinear iteration complexity bounds. Finally, Section~\ref{sec:Conclusions} summarizes the contributions and outlines directions for future research.

\subsection{Contributions and relation to the literature} \label{sec:ContReLiterature}
This subsection states our contributions and positions them within recent work. Tools from Busemann geometry, such as horospheres, horoballs, support inequalities, and an intrinsic subdifferential, are central here. We use them in a way consistent with classical Euclidean intuition to obtain a closed form projection onto horospheres and to design a projection based proximal method on Hadamard manifolds.  The horosphere serves as a Euclidean type separating set, and the projection onto it yields a Fejér type descent. In the exact case the method coincides with the classical Riemannian proximal point algorithm. First, we state our main contributions:

\begin{itemize}
  \item  We revisit the notion of subdifferential in a way that aligns with the classical Euclidean definition. We provide new properties and explicit formulas for the subdifferential, illustrated through a range of examples drawn from the literature. In addition, we introduce new classes of examples based on Busemann functions, highlighting further instances of subdifferentiable functions. 
  
 \item We present an explicit formula for the projection onto horospheres consistent with the Euclidean case, and we use it to build additional classes of subdifferentiable functions and to define a projection based proximal method.
 
  \item  We propose a Riemannian analogue of the projection based proximal method. With a relative error level strictly below one, we prove global convergence and a sublinear complexity bound in terms of a natural Busemann residual and the norm of approximate subgradients. In the exact case, the method coincides with the classical Riemannian proximal point algorithm.
\end{itemize}

We now relate our contributions to the recent literature. The closest references are \cite{Goodwin2024} and \cite{Criscitiello2025}, which inform the geometric background but pursue different algorithmic goals. The work \cite{Goodwin2024}, developed for Hadamard spaces, introduces notions and examples for the Busemann subdifferential and records properties such as inclusion relations and lack of stability under sums. We recall only what is needed for our setting, noting that Hadamard manifolds form a special subclass that allows properties depending on the Riemannian structure, for example, gradient based identities and closed form expressions. The paper \cite{Criscitiello2025} formalizes horospherical convexity and proposes first order methods, including accelerated variants, with curvature independent guarantees and frequent use of Fr\'echet mean oracles. In contrast, our work develops a proximal point approach; we use only the concepts required to support a projection based proximal scheme whose main step is a closed form projection onto a horosphere, and we prove descent, global convergence, and sublinear complexity for this method. 

\section{Basics results about Hadamard  manifolds} \label{sec:int.r}
In this section, we recall fundamental results on Hadamard manifolds, primarily drawn from \cite{Ballmann1985, BridsonHaefliger1999, Sakai1996}, and include more recent developments that will play a role in our analysis and be referenced at the appropriate points. {\it Throughout  this paper  ${\mathbb{M}}$ represents a finite dimensional Hadamard manifold}, 
 $T_p{\mathbb{M}}$ the \emph{tangent space} of ${\mathbb{M}}$ at $p$,   $T{\mathbb{M}}$ is {\itshape{tangent bundle}} of $M$ and  ${\cal X}({\mathbb{M}})$ the space of smooth vector fields on $M$. The corresponding norm associated to the Riemannian metric $\langle \cdot , \cdot \rangle$
is {represented} by $\lVert\cdot\rVert$. We use $\ell(\gamma)$ to {express} the length of a piecewise smooth curve
$\gamma\colon [a,b] \rightarrow {\mathbb{M}}$. The Riemannian distance between $p$
and $q$ in ${\mathbb{M}}$ is  denoted by $d(p,q)$,
which induces the original topology on ${\mathbb{M}}$, namely, $({\mathbb{M}}, d)$,
which is a complete metric space.  The {\it exponential mapping}
$\exp_{p}:T_{p}{\mathbb{M}} \rightarrow  {\mathbb{M}} $ is defined  by $\exp_{p}v\,=\, \gamma _{p,v}(1)$, where $\gamma _{p,v}$ is the geodesic defined by its  initial position $p$, with velocity $v$ at $p$. Hence, we have $\gamma _{p,v}(t)=\exp_{p}(tv)$. Thus, {\it we will also use the expression $\exp_{p}(tv)$ for denoting  the geodesic   $\gamma_{{p}, v}$ starting  at $p\in {\mathbb{M}}$ with velocity $v\in T_p{\mathbb{M}}$ at $p$}.  For a $p\in{\mathbb{M}}$, the exponential map $\exp_p$ is a diffeomorphism and $\log_p\colon{\mathbb{M}}\to T_p{\mathbb{M}}$ {indicates} its inverse.  In this case, $d(p,q) = \|\log_pq\|$ holds,  $d_{q}\colon{\mathbb{M}}/\{q\}\to\mathbb{R}$ defined by \(d_{q}(p):=d(p,q)\)  is $C^{\infty}$  for $q\in {\mathbb{M}}$ and  its gradient is given by ${\rm grad } \,d_{q}(p) = (-\log_pq)/d(q, p)$, for all $p\neq q$. In addition,  $d_{q}^2\colon{\mathbb{M}}\to\mathbb{R}$  defined by \(d^2_{q}(p):=d^2(p,q)\)  is $C^{\infty}$  for all $q\in {\mathbb{M}}$, and  ${\rm grad } d^2_{q}(p) = -2\log_pq$.   Let $\bar{p},\bar{q}\in {\mathbb{M}}$ and $(p_{k})_{k\in \mathbb{N}}, (q_{k})_{k\in \mathbb{N}}\subset {\mathbb{M}}$ be sequences  such that $\lim_{k\to +\infty} p_{k}=\bar{p}$ and $\lim_{k\to +\infty}q_k=\bar{q}$. Then, for any $q\in M$,  $\lim_{k\to +\infty} \log_{p_{k}}q=  \log_{\bar{p}}q$ $\lim_{k\to +\infty} \log_qp_{k}= \log_q\bar{p}$ and   $\lim_{k\to +\infty}  \log_{p_{k}}q_k=\log_{\bar{p}}\bar{q}$. For $p,q\in{\mathbb{M}}$,  {the symbol $\gamma_{pq}$ indicates} the geodesic segment  joining  $p$ to $q$, i.e., $\gamma_{pq}\colon[0,1]\rightarrow{\mathbb{M}}$ with $\gamma_{pq}(0)=p$ and $\gamma_{pq}(1)=q$.   Since $M$ is a Hadamard manifold, there exists a unique  geodesic segment  $\gamma_{pq}$  joining $p$ to $q$, its  length  is  equal $d(p,q)$, and the parallel transport along $\gamma_{pq}$ from $p$ to $q$ is denoted by $P_{pq}:T_{p}M\to T_{q}M$.  

We next recall the well-known ``cosine law'' for triangles in the $\kappa$-hyperbolic space form, see \cite[p. 138]{Sakai1996}. 
\begin{lemma}  \label{le:CosLaw1}
  Let  ${\mathbb{M}}$ be a  Hadamard manifold with curvature $\kappa< 0$ and ${x}, {y}, {z} \in {{\mathbb H}^n_{\kappa}}$. Let $\theta_{x}$  be the angle between the vectors $ \log_{{x}}{y}$ and $\log_{x}{z}$. Then, 
  \begin{equation*}
    \cosh(\sqrt{\kappa} d_{\kappa}({y}, {z}))=\cosh(\sqrt{\kappa} d_{\kappa}({x}, {y}))\cosh(\sqrt{\kappa} d_{\kappa}({x}, {z}))-\sinh(\sqrt{\kappa} d_{\kappa}({x}, {y})) \sinh(\sqrt{\kappa} d_{\kappa}({x}, {z}))  \cos \theta_{x}.
  \end{equation*}
 \end{lemma} 
We recall the well-known ``comparison theorem" for triangles in Hadamard manifolds \cite[Proposition 4.5]{Sakai1996}.
\begin{lemma}  \label{le:CosLawF}
Let ${\mathbb{M}}$ be a  Hadamard manifold. The following inequality  holds:
 \begin{equation} \label{eq:coslaw2}
d^2({x}, {y})+d^2({x},{z})-2\left\langle  \log_{{x}}{y}, \log_{{x}}{z}\right\rangle\leq d^2({y},{z}),  \qquad   \forall {x}, {y}, {z} \in {\mathbb{M}}. 
\end{equation}
In addition, if  the sectional curvature  is $K\equiv 0$ in the whole ${\mathbb{M}}$, then \eqref{eq:coslaw2} holds as equality.
\end{lemma}
A subset \(C \subset {\mathbb{M}}\) is said to be \emph{convex} if for every pair of points \(q, p \in C\), the geodesic segment 
\(
\gamma_{qp}(t) := \exp_q\bigl(t\, \log_q p\bigr),
\)
for all \(t \in [0,1]\), is entirely contained in \(C\). For every closed convex set \(C \subset {\mathbb{M}}\) and any point \(q \in {\mathbb{M}}\), there exists a \emph{unique} nearest point in \(C\), called the \emph{metric projection} of \(q\) onto \(C\) and denoted by
\[
{\cal P}_C(q) := \arg\min_{x \in C} d(q, x).
\]
The following result characterizes this projection; its proof can be found in~\cite{Ferreira2002}.

\begin{proposition}\label{prop:proj-ineq}
Let \(C \subset {\mathbb{M}}\) be closed and convex, and let \(q \in {\mathbb{M}}\). Then, for all \(p \in C\), the projection \({\cal P}_C(q)\) is unique and satisfies
\[
\big\langle \log_{{\cal P}_C(q)} q,\; \log_{{\cal P}_C(q)} p \big\rangle \leq 0.
\]
\end{proposition}

Let us recall the notion of convexity on Hadamard manifolds, fundamental to the subdifferential theory developed in \cite{Udriste1994, Ferreira2002, LiMordukhovich2011}. A set $\Omega \subset {\mathbb{M}}$ is is said to be convex, if for all $p,q\in
\Omega$ we have $\gamma_{pq}(t)\in \Omega$, for all $t\in [0,1]$.  A function
$f\colon{\mathbb{M}} \to {\mathbb{R}}$ is said to be \emph{convex (resp.
strictly convex)} if, for any $p,q\in {\mathbb{M}}$ with $p\ne
q$  the function $f\circ{\gamma_{pq}}\colon[0, 1]\to{\mathbb{R}}$ is convex (resp. strictly convex), i.e.,
$(f\circ{\gamma_{pq}})(t)\leq(1-t)f(p)+tf(q)$ (resp. ($f\circ{\gamma_{pq}})(t)<(1-t)f(p)+tf(q)$),
for all  $t\in (0,1)$. A function $f\colon{\mathbb{M}} \to{\mathbb{R}}$ is said to be
\emph{$\sigma$-strongly convex} for $\sigma > 0$ if, for any
$p,q\in {\mathbb{M}}$, the composition $f\circ{\gamma_{pq}}\colon[0, 1]\to {\mathbb{R}}$ is
$\sigma$-strongly convex, i.e.,
$(f\circ{\gamma_{pq}})(t)\leq(1-t)f(p)+tf(q)-\frac{\sigma}{2}t(1-t)d^2(q,p)$, for $t\in (0,1)$.

\begin{definition} \label{def:Sugd}	
Let  $f\colon \mathbb{M} \rightarrow {\mathbb{R}}$  be a convex function and $q\in  \mathbb{M}$.  A vector $ s\in T_{{q}}{\mathbb{M}}$ is said to be subgradient  of $f$ at ${{q}}$  if
\begin{equation} \label{eq:Sugd}
f({p}) \geq f({{q}})+\left\langle  s , \log_{ {q} } p\right\rangle, \qquad \forall p\in {\mathbb{M}}.
\end{equation}
The set of all subgradients of  $f$ at the point ${q}$ is called the subdifferential and is denoted by $\partial f({q})$.
\end{definition}
\section{The Busemann function on Hadamard manifolds} \label{sec:bfNewDef}
In this section, we revisit the concept of the Busemann function on Hadamard manifolds, providing essential notations and discussing its properties. Given that our definition is slightly more general than the conventional one used in the literature, we have opted to include streamlined proofs for selected results. This concise review encapsulates pertinent findings that play an essential role in optimization. For additional information on this topic, see, for example,  \cite{Sakai1996}.

Let  ${\mathbb{M}}$  be a Hadamard manifold and $d$ the Riemannian distance.  The {\it Busemann function} associated with a base point  ${q} \in  {\mathbb{M}}$  and   ${v} \in T_{q} {\mathbb{M}}$ is defined as follows 
\begin{equation} \label{eq:defBFNDef}
B_{{q}, v}(p) \coloneqq \lim_{t \to+ \infty} \left(d\left({p}, \exp_{q}(tv)\right)-\|v\| t\right),     \qquad  \forall p\in {\mathbb{M}}. 
\end{equation}
Note that, for $v=0$ in \eqref{eq:defBFNDef} we have $B_{{q}, 0}(p)=d({q}, p)$. In addition, it follows from triangular inequality  that 
\begin{equation} \label{eq:defFIWDef2b}
|B_{{q}, v}(p)|\leq d({q}, p),    \qquad  \forall q, p \in {\mathbb{M}}, \quad \forall v\in T_{q}{\mathbb{M}}.
\end{equation}
Since the distance function  is  Lipschitz continuous with constant \(L=1\),  the Busemann function  $B_{q,v}$ is also  {\it Lipschitz continuous with constant \(L=1\)}, i.e., 
\begin{equation}\label{eq:1Lip}
      \bigl|B_{q,v}(x)-B_{q,v}(y)\bigr| \leq d(x,y),  \qquad\forall\,x,y\in\mathbb{M}.
\end{equation}
Let $p \in {\mathbb{M}}$ and $[0, +\infty)\ni t\mapsto B_{{q}, v}(t,p) :=d\left({p}, \exp_{q}(tv)\right)-\|v\| t$  be a auxiliary  function. Then,  using     \cite[Lemma II.8.18 (1), p. 268]{BridsonHaefliger1999} together with   the  triangle inequality we conclude  that  
\begin{equation*}
-d(p,q)\leq B_{{q}, v}(t,p)\leq B_{{q}, v}(\tau,p), \qquad  \tau\leq t,    \qquad  \forall p \in {\mathbb{M}}.
\end{equation*}

\begin{remark} \label{re:bfn1}
Conventionally, the  Busemann function is defined for $\|v\|=1$ see, for example \cite[Definition 8.17, p. 268]{BridsonHaefliger1999} and  \cite[p. 212]{Sakai1996}.  However,  it follows from \eqref{eq:defBFNDef} that  Busemann function is \emph{scale-invariant} in its direction, namely \(B_{q,\lambda v}=B_{q,v}\),  for every \(\lambda>0\). In particular,  we have  $B_{{q}, v}= B_{{q}, \frac{v}{\|v\|}}$, for all $v\neq 0$. For the purpose of this study, it is beneficial to extend its definition to include the case of $v=0$. 
\end{remark}

In the following lemma, we provide two distinct formulas for computing the gradient of the Busemann function, a fundamental object in the study of Hadamard manifolds that plays a central role in the present work. The first formula appears implicitly in the proof of \cite[Lemma 4.12, p.~231]{Sakai1996}. 

\begin{lemma} \label{le:CharactBusFunc}
Let ${q} \in {\mathbb{M}}$ be a base point in a Hadamard manifold, and let ${v} \in T_{q} {\mathbb{M}}$ be  such that $v\neq 0$. Then the Busemann function $B_{q, v}$ is convex. Moreover, $B_{q,v}$ is differentiable on $\mathbb{M}$, and its gradient vector field is given by either of the following formulas:
\begin{equation}\label{eq:gradbf2}
{\rm grad} B_{q, v}(p) = -\lim_{t \to \infty} \frac{\log_{p}(\exp_{q}(t v))}{d(p, \exp_{q}(t v))}= -\frac{1}{\|v\|} \lim_{t \to \infty} \frac{\log_{p}(\exp_{q}(t v))}{t}, \qquad \forall p \in {\mathbb{M}}.
\end{equation}
In addition, the gradient vector field ${\rm grad} B_{q, v}$ is continuous and satisfies $\|{\rm grad} B_{q, v}(p)\| = 1$ for all $p \in {\mathbb{M}}$. In particular, at the base point, we have
\(
{\rm grad} B_{q, v}(q) = -{v}/{\|v\|}.
\)
\end{lemma}
We remark that the lemma above establishes fundamental properties of the Busemann function that will be instrumental in our subsequent analysis.  Since  $B_{{q}, 0}=d({q}, p)$, we also obtain that  $B_{{q}, 0}$ is  convex  and continuously  differentiable with  the gradient vector field ${\rm grad }  B_{{q}, 0}$ satisfying  $\|{\rm grad }  B_{{q}, 0}(p)\|=1$, for all $p\neq {q}$.   The following lemma, whose proof is straightforward and thus omitted, in particular implies that the Busemann function $B_{{q}, v}$ is linear along the geodesic $t \mapsto \exp_q(tv)$ that defines it.

\begin{lemma} \label{eq:pbfu}
Let \(q\in\mathbb{M}\) be fixed and let \(u,v\in T_q\mathbb{M}\). Then,  \(B_{q,-v}\!\bigl(\exp_q(\tau v)\bigr)=\tau\|v\|\), for all \(\tau\in \mathbb{R}\). Consequently, the Busemann function \(B_{q,v}\) is unbounded both above and below.
\end{lemma}

\begin{lemma}\label{le:cbf} 
Let $\bar{q} \in {\mathbb{M}}$ and $\bar{v} \in T_{\bar{q}} {\mathbb{M}}$. Consider sequences $(q_k)_{k \in \mathbb{N}} \subset {\mathbb{M}}$ and $(v_k)_{k \in \mathbb{N}}$ with $v_k \in T_{q_k} {\mathbb{M}}$, satisfying  $\lim_{k \to +\infty} q_k = \bar{q}$, and $\lim_{k \to +\infty} v_k = \bar{v}$. Then, 
\(\lim_{k \to +\infty} B_{q_k,v_k}(p) = B_{\bar{q}, \bar{v}}(p),\) for all $p \in {\mathbb{M}}$.
\end{lemma}
The next identity can be employed as a practical tool for computing the Busemann function. 
\begin{proposition}  \label{pr:afcbf}
Let ${q} \in  {\mathbb{M}}$ be a base point and    ${v} \in T_{q} {\mathbb{M}}$ with  $v\neq 0$.  Then, there holds
\begin{equation*} 
B_{{q}, v}(p) = \lim_{t \to+ \infty} \frac{d^2({p}, \exp_{q}(tv))-(\|v\|t)^2}{2\|v\|t},     \qquad  \forall p\in {\mathbb{M}}. 
\end{equation*}
\end{proposition} 

We conclude this  section  with an inequality comparing the Busemann function to the Riemannian pairing at the base point \(q\). It is a support-type bound along the ray in direction \(v\), becomes tight in the zero-curvature case, and will be used later to relate Busemann subgradients to classical subgradients.

\begin{lemma} \label{le:Inbusfunc}
Let ${\mathbb{M}}$ be a Hadamard manifold. Then the Busemann function $B_{q, v}$ defined in \eqref{eq:defBFNDef}, associated with a base point $q \in {\mathbb{M}}$ and a direction $v \in T_q{\mathbb{M}}$, satisfies the inequality
\begin{equation} \label{eq:mainineq} 
 -\left\langle v, \log_q p \right\rangle\leq \|v\| \, B_{q, v}(p), \qquad \forall p \in {\mathbb{M}}.
\end{equation}
Moreover, if  the sectional curvature  is $K\equiv 0$ in whole ${\mathbb{M}}$, then inequality \eqref{eq:mainineq}  holds as equality
$\|v\| \, B_{q, v}(p) = -\left\langle v, \log_q p \right\rangle$, for all  $p \in {\mathbb{M}}$.
\end{lemma}
\begin{proof}{Proof}
It is immediate to conclude that \eqref{eq:mainineq}  holds for $v=0$. Now, we assume that $v\neq 0$. It follows from Lemma~\ref{le:CosLawF}, applied  to the triangle with vertexes $x={q}$, $y=p$ and $z=\exp_{q}(tv)$,  that  
 \begin{equation*}
d^2({q},{p})+ d^2({q}, {\exp_{q}(tv)})-2\left\langle \log_{{q}}{p}, \log_{{q}}{\exp_{q}(tv)}\right\rangle\leq d^2({p},{\exp_{q}(tv)}).
\end{equation*}
Since  $\log_{{q}}{\exp_{q}(tv)})=tv$ and $d({q}, {\exp_{q}(tv)})=t\|v\|$, it  follows from the last inequality that 
\begin{equation*} 
d^2({q},{p})-2\left\langle  tv, \log_{{q}}{p}\right\rangle\leq d^2({p},{\exp_{q}(tv)}-(\|v\| t)^2
\end{equation*}
After performing some  algebraic manipulations, the last inequality can be expressed as follows
$$
\frac{d^2({q},{p})}{2t}- \left\langle  v, \log_{{q}}{p}\right\rangle \leq \|v\| \frac{d^2({p}, \exp_{q}(tv))-(\|v\|t)^2}{2\|v\|t} 
$$
By taking the limit as $t$  tends to $+\infty$,  and employing Proposition~\ref{pr:afcbf}, we establish the inequality \eqref{eq:mainineq}.
\end{proof}

In general, when the sectional curvature of $\mathbb{M}$ is nonzero, the inequality in \eqref{eq:mainineq} cannot hold with equality. This is because the function on the right-hand side is neither convex, concave, nor quasi-convex, see \cite{Kristaly2016}.


\subsection{Examples of Busemann functions}

In this section, we present explicit formulas and basic properties of Busemann functions in three canonical Hadamard settings, namely flat manifolds with zero sectional curvature, hyperbolic spaces of constant curvature $\kappa<0$, and the manifold of symmetric positive definite matrices endowed with the affine–invariant metric. We begin with the flat case, where the Busemann function is affine, proceed to the hyperbolic spaces, and conclude with the SPD manifold.

\begin{example} \label{ex:defBFNDefK0}
Let ${\mathbb{M}}$ be a Hadamard manifold,   ${q} \in  {\mathbb{M}}$ be  a base point and   ${v} \in T_{q}{\mathbb{M}}$.  If the sectional curvature of  ${\mathbb{M}}$  remains identically zero throughout the entire manifold, denoted as   $K\equiv 0$.  Then, the  Busemann function $B_{{q}, v}$  is given by 
\begin{equation} \label{eq:defBFNDefK0}
B_{{q}, v}(p):= \begin{cases}
-\left\langle  \frac{v}{\|v\|}, \log_{{q}}{p}\right\rangle,  &   \qquad  \forall v\in T_{q}{\mathbb{M}},~ v\neq 0, \\
d({q},p),  &  \qquad v=0.
\end{cases}
\end{equation}
 In fact, if $K\equiv 0$ in whole ${\mathbb{M}}$, then   $\|v\| \, B_{q, v}(p) = \left\langle v, \log_q p \right\rangle$, for all  $p \in {\mathbb{M}}$, which together with $B_{{q}, 0}(p)=d({q}, p)$ implies \eqref{eq:defBFNDefK0}. 
 In particular, for ${\mathbb{M}}={{\mathbb R}}^n$ we have $\log_{{q}}{p}=p-{q}$ and $d({q}, p)=\|p-{q}\|$, giving that Busemann functions are linear in this setting.
\end{example}

In the following example, we specifically provide explicit formulas  on a $k$-hyperbolic space form.
\begin{example}   \label{ex:hsfbc}
 References to this example include  \cite{BenedettiPetronio1992, Boumal2020, Ratcliffe2019}. For a given $\kappa>0$,  the {\it $n$-dimensional  ${\kappa}$-hyperbolic space form} and its {\it tangent hyperplane at a point $p$} are denoted  by
\begin{equation*}
{\mathbb H}^{n}_{\kappa}:=\Big\{ p\in {{{\mathbb R}}^{n+1}}:~\langle p, p\rangle=-\frac{1}{\kappa}, ~p^{n+1}>0\Big\}, \qquad T_{p}{{\mathbb H}^n_{\kappa}}:=\left\{v\in {{{\mathbb R}}^{n+1}}\, :\, 
\left\langle p, v \right\rangle=0\right\},
\end{equation*}
where,  $\left\langle \cdot , \cdot \right\rangle$ is the   {\it Lorentzian inner product}   $\left\langle x, y\right\rangle:= x^{{T}}{\rm J}y$ and ${\rm J}:={\rm diag}(1, \ldots,1, -1) \in {{\mathbb R}}^{(n+1)\times (n+1)}.$     The {\it intrinsic distance on the  $\kappa$-hyperbolic space form} between two  points $p, q \in {{\mathbb H}^n_{\kappa}}$  is  given~by
\begin{equation*}
d_{\kappa}(p, q):=\frac{1}{\sqrt{\kappa}}{\rm arcosh} (-\kappa \left\langle p , q\right\rangle).
\end{equation*}
Let ${q} \in  {\mathbb H}^{n}_{\kappa}$ be  a base point,  ${v} \in T_{q}{\mathbb H}^{n}_{\kappa}$ .  Then, the  Busemann function $B_{{q}, v}:   {\mathbb H}^{n}_{\kappa} \to {{\mathbb R}}$  is given by 
\begin{equation} \label{eq:BFkSF}
B_{{q}, v}(p):= \begin{cases}
 \frac{1}{\sqrt{\kappa}}\ln\Big( - \big\langle p, \kappa\,{q}+ {\sqrt{\kappa}}\,\frac{1}{\|v\|}{v}\big\rangle \Big),  &   \qquad  \forall v\in T_{q}{\mathbb{M}},~ v\neq 0, \\
d({q},p),  &  \qquad v=0.
\end{cases}
\end{equation}
and its gradient, for $v\in T_{q}{\mathbb{M}}$ with  $v\neq 0$ is given by 
\begin{equation*} 
{\rm grad }  B_{{q}, v}(p)=\frac{1}{\sqrt{\kappa}}\frac{1}{\big\langle p, \kappa\,{q}+ {\sqrt{\kappa}}\frac{1}{\|v\|}{v}\big\rangle}\Big(\kappa\,{q}+
{\sqrt{\kappa}}\frac{1}{\|v\|}{v}+ \kappa \big\langle \kappa\,{q}+
{\sqrt{\kappa}}\frac{1}{\|v\|}{v}, p\big\rangle \,p \Big).
\end{equation*}
\end{example}

For additional examples of Busemann functions, see  \cite{Bento2023, BridsonHaefliger1999, hirai2023}.

\section{B-subdifferential on  Hadamard manifold} \label{sec:NewDefSub}
In this section, we introduce a notion of subgradient adapted to the geometry of Busemann functions on Hadamard manifolds. This \emph{Busemann subdifferential} (or \emph{B-subdifferential}) essentially coincides with the construction in \cite{Goodwin2024}; see also \cite{Criscitiello2025}. The motivation is that the usual linear support built from the logarithmic map need not be convex when curvature is nonzero, whereas a Busemann-based support inequality furnishes an appropriate convex lower bound. Our definition agrees with the classical subdifferential in the flat case, and we establish basic properties: inclusion into the classical subdifferential, a simple chain rule for nondecreasing convex scalars, and max-type constructions under an alignment condition. This framework will be used later for projection statements onto horospheres and for the analysis of a Busemann-based hybrid projection–proximal method.

The mapping on the right-hand side of \eqref{eq:Sugd} in Definition~\ref{def:Sugd}, namely
\begin{equation} \label{eq:Sugdlhs}
{\mathbb{M}} \ni p\mapsto f(q)+\left\langle s , \log_{q} p\right\rangle,
\end{equation}
has often been treated as convex on Hadamard manifolds; as noted in \cite{Kristaly2016}, this unwarranted convexity assumption (generally false when the curvature is nonzero) has led to incorrect statements in the literature. If the sectional curvature of $\mathbb{M}$ is identically zero, then \eqref{eq:Sugdlhs} is linear, hence convex. When the curvature is nonzero, \eqref{eq:Sugdlhs} is, in general, neither concave nor convex nor quasi-convex. To address this, we adopt a subdifferential based on the Busemann function.
\begin{definition} \label{def:Bsub}
Let  $f\colon \mathbb{M} \rightarrow {\mathbb{R}}$  be a convex function and ${q}\in {\rm dom} f$.  A vector $s\in T_{{q}}{\mathbb{M}}$ is said to be $B$-subgradient  of $f$ at ${{q}}$  if
\begin{equation} \label{eq:bsudiff}
f({p}) \geq f({{q}})+\|s\| B_{{q},-s}({p}), \qquad \forall p\in {\mathbb{M}}.
\end{equation}
The set of all $B$-subgradients of the function $f$ at ${q}$ is called the  $B$-subdifferential and is denoted by $\partial^{b}f({q})$.
\end{definition}
By Example~\ref{ex:defBFNDefK0} and Definition~\ref{def:Bsub}, in the  zero-curvature one has $\partial f(q)=\partial^{\mathrm{b}}f(q)$ for every $q\in\mathbb{M}$. Thus, Definition~\ref{def:Bsub} agrees with the classical subdifferential on Hadamard manifolds as in Definition~\ref{def:Sugd}; see also \cite[Definition~4.3, p.~73]{Udriste1994}. In particular, it reduces to the usual Euclidean subgradient in $\mathbb{R}^n$. In general curvature, although $\|v\|\,B_{q,v}(p)$ need not equal $\langle v,\log_q p\rangle$, we will show that the inclusion $\partial^{\mathrm{b}}f(q)\subseteq\partial f(q)$ holds on any Hadamard manifold. Moreover, Lemma~\ref{le:CharactBusFunc} implies that, for fixed $q$ and $s$, the mapping
$ p \,\mapsto\, f(q)+\|s\|\,B_{q,-s}(p)$, $p\in\mathbb{M},$
is convex in $p$. In the following, we establish a basic property of the B-subdifferential; this also explains why, in the definition of the Busemann function \eqref{eq:defBFNDef}, we allow the zero vector.
 
\begin{proposition} \label{pr:ocp}
Let $f\colon {\mathbb{M}} \rightarrow {\mathbb{R}}$ be convex.  The point ${q}$ is a minimizer  of $f$ if and only if $0\in  \partial^{b} f({q})$.
\end{proposition}
\begin{proof}{Proof}
Assume that ${q}$ is a minimizer of $f$. Thus, we have $f(p)\geq f({q})= f({q})+\|0\|B_{{q},0}({p})$, for all $p\in {\mathbb{M}}.$
Therefore, $0\in  \partial^{b} f({q})$. Reciprocally, assume that  $0\in  \partial^{b} f({q})$. Thus,   Definition~\ref{def:Bsub} implies that   $f(p)\geq f({q})+\|0\|B_{{q},0}({p})= f({q})$, for all $p\in {\mathbb{M}}$. Therefore, ${q}$ is a minimizer  of $f$.
\end{proof}

In the next result, we formally prove that the B-subdifferential introduced in Definition~\ref{def:Bsub}  is a subset of the usual subdifferential from Definition~\ref{def:Sugd}. 
\begin{proposition} \label{pr:BSubset}
Let $\mathbb{M}$ be a Hadamard manifold and let $f\colon {\mathbb{M}} \rightarrow {\mathbb{R}}$ be a convex function. Then, for every $q \in {\rm int}\operatorname{dom} f$, the B-subdifferential of $f$ at $q$ is a subset of  the classical subdifferential, i.e.,
\[
\partial^{\mathrm{b}} f(q) \subseteq \partial f(q).
\]
\end{proposition}

\begin{proof}{Proof}
Let $s\in \partial^{b}f({q})$. Then, by using \eqref{eq:bsudiff} we have  \(f({p}) \geq f({{q}})+\|s\| B_{{q},-s}({p})\),  for all \(p\in {\mathbb{M}}\). On the other hand, Lemma~\ref{le:Inbusfunc} implies that  \( \left\langle s, \log_q p \right\rangle\leq \|s\| \, B_{q, -s}(p)\), for all \(p \in {\mathbb{M}}\). Combining two previous inequalities,  we obtain   \(f({p})\geq f({{q}})+\left\langle  s , \log_{ {q} } p\right\rangle\), for all \(p\in {\mathbb{M}}.\)  Therefore, using
Definition~\ref{def:Sugd},  we conclude that  $s\in \partial  f({q})$, which proves
the inclusion. 
\end{proof}

Note that in the flat case, the $B$-subdifferential coincides with the classical subdifferential. 
This observation will be useful later when comparing our results with those in Euclidean convex analysis.

\begin{remark}
If the sectional curvature of $\mathbb{M}$ is identically zero, denoted $K\equiv 0$, then, by combining Definitions~\ref{def:Sugd}  and \ref{def:Bsub}   with Example~\ref{ex:defBFNDefK0}, we obtain that for every $q \in \mathbb{M}$ the $B$-subdifferential of $f$ at $q$ coincides with the classical subdifferential, that is,
\(
\partial^{\mathrm{b}} f(q) = \partial f(q).
\)
\end{remark}

As the next example shows, the inclusion in Proposition~\ref{pr:BSubset} may become an equality for certain classes of functions on any Hadamard manifold. These examples are special cases of constructions from \cite{Goodwin2024} (see also \cite{Criscitiello2025}); we present them here because, in the Riemannian setting of Hadamard manifolds, the structure is more explicit, the geometry is clearer, and one can derive closed-form expressions for the subdifferential together with additional properties.

\begin{example} \label{ex:suddist}
Fix ${q}\in {\mathbb{M}}$ and consider the distance function $p\mapsto d_{{q}}(p):=d({q},  p)$, where $d$ is the  Riemannian distance in ${\mathbb{M}}$. Then,  the  B-subdifferential  of \( d_{{q}}\) is given by

\begin{equation} \label{eq:BusSub}
\partial^{b} d_{{q}}(p):= \begin{cases}
\{s\in T_{q}{\mathbb{M}}:~\|s\|\leq 1\}  &   \qquad  p=q, \\
\{-{\log_{p}q}/{d(p,q)}\},  &  \qquad p\neq q.
\end{cases}
\end{equation}
 In fact, first assume that \(p=q\) and take  any \(s\in T_q\mathbb{M}\) with \(\|s\|\le 1\). It follows from \eqref{eq:defFIWDef2b} that \(|B_{q,-s}(p)|\le d_q(p)\),  for every \(p\in\mathbb{M}\).
Hence, taking into account that  \(d_q(q)=0\) and \(\|s\|\le 1\) we have 
\[
d_q(p)-d_q(q)=d_q(p)\geq |B_{q,-s}(p)|\geq \|s\|\,B_{q,-s}(p),\qquad\forall\,p\in\mathbb{M}.
\]
Thus,  it follows from Definition~\ref{def:Bsub}  that \(s\in\partial^{b} d_q(q)\). Therefore,
\(
\bigl\{\,s\in T_q\mathbb{M} : \|s\|\le 1\bigr\}\subset \partial^{b} d_q(q), 
\)
and  taking into account  that  the usual subdifferential  is   given by $\partial  d_{{q}}({q})=\{s\in T_{q}{\mathbb{M}}:~\|s\|\leq1 \}$, see for example \cite[Theorem 5.3]{LiMordukhovich2011},  we conclude that $\partial^{b}  d_{{q}}({q})=\{s\in T_{q}{\mathbb{M}}:~\|s\|\leq1 \}$.

For every $p\neq q$ the Busemann subdifferential of $d_q$ at ${p}$ is  \(\partial^{b}d_q({p})=\{-{\log_{{p}}q}/{d({p},q)}\}\). Indeed, for simply the notations we set \(v:=-({\log_{{p}}q})/{d({p},q)}\).   Since \(B_{{p},-v}\) is  Lipschitz continuous with constant \(L=1\),  we have 
\begin{equation} \label{eq:cdsg}
B_{{p},-v}({x})- B_{{p},-v}(q)\leq d_{q}({x}), \qquad \forall {x}\in\mathbb{M}.
\end{equation} 
On the other hand,  due to   \(v=-({\log_{{p}}q})/{d({p},q)}\) we obtain that  \(q=\exp_{{p}}(-d({p},q)v)\).  Hence,  due to  $q$ lies on the geodesic $t\mapsto \exp_{p}(-t v)$ we have \(d({q}, \exp_{{p}}(-tv))=|t-d({p}, q)|\). Thus, taking into account that \(\|v\|=1\), we conclude that 
\begin{equation*} 
B_{{{p}}, -v}(q) \coloneqq \lim_{t \to+ \infty} \left(d({q}, \exp_{{p}}(-tv))-\|v\| t\right)=  \lim_{t \to+ \infty} \left( |t-d({p}, q)| -t\right)=-d_{q}({p}).
\end{equation*}
Substituting, the last inequality into \eqref{eq:cdsg} we obtain that 
\[
 d_{q}({x})\geq d_{q}({p})+ B_{{p},-v}({x}), \qquad \forall {x}\in\mathbb{M}.
\]
Therefore,  \(v:=-({\log_{{p}}q})/{d({p},q)}\in  \partial^{b} d_{{q}}({{p}})\). Because, Proposition~\ref{pr:BSubset} implies that \(\partial^{\mathrm{b}} d_{q}(p) \subseteq \partial d_{q}(p)\)   and  considering that the usual subdifferential is given by  \(\partial d_q({p})=\{-{\log_{{p}}q}/{d({p},q)}\}\), we conclude that \(\partial^{b}d_q({p})=\{-{\log_{{p}}q}/{d({p},q)}\}\) and \eqref{eq:BusSub} holds.
\end{example}

\begin{example}
Let \( {\mathbb{M}} \) be a Hadamard manifold, let \( q \in {\mathbb{M}} \), and let \( v \in T_q\mathbb{M} \setminus \{0\} \). 
For a given   \( p \in {\mathbb{M}} \)  the  gradient  vector \(s:={\rm grad} B_{q, v}(p)\) is a $B$-subgradient  of  Busemann function \(B_{q,v}\) at \(p\), i.e., 
\[
B_{q,v}(x) \geq B_{q,v}(p) + \|s\| B_{p, -s}(x), \qquad \forall x \in {\mathbb{M}}.
\]
In fact, consider the geodesic ray $\gamma(t)=\exp_q(tv)$. If $p=q$, then $B_{q,v}(q)=0$. By Remark~\ref{re:bfn1}, $B_{q,v}=B_{q,\,v/\|v\|}$, so we may assume $\|v\|=1$. In addition, Lemma~\ref{le:CharactBusFunc} implies that  ${\rm grad}\,B_{q,v}(q)=-v/\|v\|$. Therefore,  the subgradient inequality at $q$ holds with $s=-v/\|v\|$. Hence $-v/\|v\|\in\partial^{\,b} B_{q,v}(q)$.  Now, we assume  that \(p\neq q\). Thus, we apply Lemma~\ref{le:CharactBusFunc} to define the vector 
\[
s:={\rm grad} B_{q,v}(p)
   =-\lim_{t\to\infty}\frac{\log_{p}\!\bigl(\gamma(t)\bigr)}
                               {d\!\bigl(p,\gamma(t)\bigr)}
   \in T_{p}\mathbb{M} .
\]
and note that  \(\lVert s\rVert=1\). For each real number \(t>0\) introduce the unit vector
\[
w(t):=-\frac{\log_{p}\!\bigl(\gamma(t)\bigr)}
               {d\!\bigl(p,\gamma(t)\bigr)}\in T_{p}\mathbb{M}.
\]
Example~\ref{ex:suddist} applied to the distance function
\(x\mapsto d\!\bigl(x,\gamma(t)\bigr)\) shows that \(w(t)\) lies in the
Busemann subdifferential of that distance at the point \(p\).
Consequently,
\[
d\!\bigl(x,\gamma(t)\bigr)
\geq 
d\!\bigl(p,\gamma(t)\bigr)+B_{p,-w(t)}(x), \qquad \forall x\in\mathbb{M}.
\]
Subtracting the constant \(\lVert v\rVert\,t\) yields
\begin{equation}\label{eq:t-inequality}
d\!\bigl(x,\gamma(t)\bigr)-\lVert v\rVert t \geq  d\!\bigl(p,\gamma(t)\bigr)-\lVert v\rVert t +B_{p,-w(t)}(x), \qquad \forall x\in\mathbb{M}..
\end{equation}
Since Lemma~\ref{le:cbf} implies that  \(\lim_{t\to\infty}B_{p,-w(t)}(x)=B_{p,-s}(x)\), for every \(x\in\mathbb{M}\). Therefore, taking the limit as \(t\) goes to \(+\infty\) in inequality~\eqref{eq:t-inequality} and using definition of Busemann function, we obtain
\[
B_{q,v}(x)\geq  B_{q,v}(p)+B_{p,-s}(x), \qquad \forall x\in\mathbb{M}.
\]
Since \(\lVert s\rVert=1\), multiplying the last term by
\(\lVert s\rVert\) does not change the expression, and the desired
inequality follows. As a consequence,  we have \( \partial^{\mathrm{b}} f(p) = \partial f(p)\), for all \(p\in  {\mathbb{M}} \).
\end{example}

Next, we define a class of functions that will be useful for presenting examples of $B$-subdifferentiable functions and for simplifying subsequent constructions.
\begin{definition}\label{def:AlignedBClass}
A class of geodesically convex and $B$-subdifferentiable functions $\{g_1,\ldots,g_m\}$ with $g_j:\mathbb M\to\mathbb R$ is said to belong to the
\emph{aligned $B$-support class at $q\in\mathbb M$} if there exist a vector $u\in T_q\mathbb M$ and scalars $\alpha_j\ge 0$
such that
\begin{equation}\label{eq:AlignedBClassPoint}
g_j(p)\ \ge\ g_j(q)\ +\ \alpha_j\,B_{q,-u}(p)\qquad \forall\,p\in\mathbb M,\ \ j=1,\ldots,m.
\end{equation}
We say that $\{g_1,\ldots,g_m\}$ belongs to the \emph{aligned $B$-support class on $\mathbb M$} if
\eqref{eq:AlignedBClassPoint} holds \emph{for every} $q\in\mathbb M$.
\end{definition}

In the following, we present a canonical class of functions satisfying Definition~\ref{def:AlignedBClass}.

\begin{example}\label{ex:common-scalar-backbone}
Let $h:\mathbb{M}\to\mathbb{R}$ be geodesically convex and $B$-subdifferentiable, and let $\psi_j:\mathbb{R}\to\mathbb{R}$ be convex and nondecreasing for $j=1,\dots,m$. Define $g_j:\mathbb{M}\to\mathbb{R}$ by 
\[
g_j(p):=w_j \psi_j\!\big(h(p)\big) + {a}_j,\qquad j=1,\dots,m,
\]
where $w_j\ge 0$ and ${a}_j\in\mathbb R$. 
Then , the class $\{g_1,\ldots,g_m\}$ belongs to the aligned $B$-support class on $\mathbb M$, i.e., it satisfies Definition~\ref{def:AlignedBClass}. Indeed, fix $q\in\mathbb M$ and choose $u\in\partial^{\,b}h(q)$. By  Definition~\ref{def:Bsub}, 
\begin{equation} \label{eq:fibe}
h(p)-h(q)\ \ge\ \|u\|\,B_{q,-u}(p)\qquad \forall p\in\mathbb{M}.
\end{equation} 
Since $\psi_j$ is nondecreasing, for any $\theta_j\in\partial\psi_j\big(h(q)\big)$ we have $\theta_j\ge 0$, and since $\psi_j$ is convex we obtain
\[
\psi_j(t)\ \ge\ \psi_j\big(h(q)\big)+\theta_j\,(t-h(q))\qquad \forall t\in\mathbb{R}.
\]
Applying this with $t=h(p)$ and combining with the support inequality for $h$, and using \eqref{eq:fibe}  yields
\begin{equation*}
g_j(p)-g_j(q)= w_j\big(\psi_j(h(p))-\psi_j(h(q))\big)\ge w_j\,\theta_j\,\big(h(p)-h(q)\big)\ge w_j\,\theta_j\,\|u\|\,B_{q,-u}(p).
\end{equation*}
Thus \eqref{eq:AlignedBClassPoint} holds with the same direction $u$ for all $j$ and
\(
\alpha_j:=w_j\,\theta_j\,\|u\|\ \ge\ 0.
\)
This completes the argument.
\end{example}
The next result is a \emph{chain rule} for the $B$-subdifferential under an alignment hypothesis,  namely, the functions under consideration belong to the aligned $B$-support class on $\mathbb{M}$ in Definition~\ref{def:AlignedBClass}.

\begin{proposition}\label{prop:aligned-unifiedf}
 Suppose $g_1,\dots,g_m:\mathbb M\to\mathbb R$ are geodesically convex and $B$-subdifferentiable. 
Assume that the family $\{g_1,\ldots,g_m\}$ belongs to the aligned $B$-support class at $q\in\mathbb M$ in the sense of Definition~\ref{def:AlignedBClass}, with a fixed vector $u\in T_q\mathbb M$ satisfying $\|u\|>0$ and coefficients $\alpha_j\ge0$.
Let $\Omega\subset \mathbb{R}^{m}$ be a convex set  and  $\Phi:\Omega \to\mathbb{R}$ be a convex function and nondecreasing  in each coordinate. Assume that $\mathbb M \ni p\mapsto (g_1(p),\dots,g_m(p))\in \Omega$ and define
\begin{equation}\label{eq:f-Phi-finite}
    f(p)\ :=\ \Phi\big(g_1(p),\ldots,g_m(p)\big), \qquad \forall p\in\mathbb{M}.
\end{equation}
Then, for every $(w_1,\ldots,w_m)\in\partial\Phi\big((g_1(q),\ldots,g_m(q))\big)$, there holds
\begin{equation}\label{eq:Phi-chain-finite}
\Big(\sum_{j=1}^m w_j\,\alpha_j\Big)\,\frac{u}{\|u\|}\ \in\ \partial^{\,b} f(q).
\end{equation}
In particular, $f$ is $B$-subdifferentiable at $q$, and
\[
\partial^{\,b} f(q)\ \supseteq\ \Big\{\big(\sum_{j=1}^m w_j\alpha_j\big)\frac{u}{\|u\|}\ :\ (w_1,\ldots,w_m)\in\partial\Phi\big((g_1(q),\ldots,g_m(q))\big)\Big\}.
\]
\end{proposition}
\begin{proof}{Proof}
Take $w=(w_1,\ldots,w_m)\in\partial\Phi((g_1(q),\ldots,g_m(q)))$.  By convexity of $\Phi$ we have
\[
\Phi(y)\ \ge\ \Phi(x)\ +\ \sum_{j=1}^m w_j\,(y_j-x_j), \qquad \forall  x, y\in\mathbb{R}^m
\]
Choosing $x_j=g_j(q)$ and $y_j=g_j(p)$ and using definition of $f$ in \eqref{eq:f-Phi-finite} we conclude that 
\begin{equation}\label{eq:icrf}
f(p)-f(q)\ \ge\ \sum_{j=1}^m w_j\,\big(g_j(p)-g_j(q)\big), \qquad \forall\,p\in\mathbb M.
\end{equation}
Because $\Phi$ is nondecreasing in each coordinate, every subgradient is coordinatewise nonnegative. Thus, by the aligned $B$-support hypothesis at $q$ in Definition~\ref{def:AlignedBClass}, for all $p\in\mathbb M$ and all $j$, we have 
\(
g_j(p)-g_j(q)\ \ge\ \alpha_j\,B_{q,-u}(p).
\)
Combining with \eqref{eq:icrf} gives
\[
f(p)-f(q)\ \ge\ \Big(\sum_{j=1}^m w_j\,\alpha_j\Big)\,B_{q,-u}(p), \qquad \forall\,p\in\mathbb M.
\]
Define the vector 
\[
s\ :=\ \Big(\sum_{j=1}^m w_j\,\alpha_j\Big)\,\frac{u}{\|u\|}\ \in T_q\mathbb M.
\]
Then $\|s\|=\sum_{j=1}^m w_j\,\alpha_j$, and by the scale invariance of the Busemann function in Remark~\ref{re:bfn1}, we have 
$B_{q,-s}=B_{q,-u}$. Therefore,
\[
f(p)-f(q)\ \ge\ \|s\|\,B_{q,-s}(p)\qquad \forall\,p\in\mathbb M,
\]
which is exactly $s\in\partial^{\,b}f(q)$ by Definition~\ref{def:Bsub}. This proves \eqref{eq:Phi-chain-finite}.  The set-inclusion statement follows immediately by ranging over all $w\in\partial\Phi\big((g_1(q),\ldots,g_m(q))\big)$, which concludes the proof.
\end{proof}

In the following  we obtain as application of Proposition~\ref{prop:aligned-unifiedf}  the chain rule  stated in   \cite[Proposition~3.4]{Goodwin2024}; see also \cite[Proposition~2(vii)]{Criscitiello2025}.

\begin{corollary}\label{cor:Bchain}
Let $\mathbb M$ be a Hadamard manifold, let $g:\mathbb M\to\mathbb R$ be geodesically convex and $B$-subdifferentiable, and let $\varphi:\mathbb R\to\mathbb R$ be convex and nondecreasing.
Define $f:=\varphi\circ g$. Then $f$ is $B$-subdifferentiable. More precisely, for any $q\in\mathbb M$, if \(s\in\partial^{\,b}g(q)\) and \(\zeta\in\partial\varphi\big(g(q)\big)\), hence $\zeta\ge0$, then  there holds
\(
 \zeta\,s\in\partial^{\,b}f(q).
\)
\end{corollary}

\begin{proof}{Proof}
Fix $q\in\mathbb M$. If $s=0$, then $\zeta s=0\in\partial^{\,b}f(q)$ trivially. Assume $s\neq0$. 
By Definition~\ref{def:Bsub}, 
\[
g(p)-g(q)\ \ge\ \|s\|\,B_{q,-s}(p)\qquad(\forall p\in\mathbb M).
\]
Set $u:=s$. By Remark~\ref{re:bfn1} we have  $B_{q,-u}=B_{q,-s}$. Thus
\(
g(p)-g(q)\ \ge\ \|s\|\,B_{q,-u}(p).
\)
This is \eqref{eq:AlignedBClassPoint} with $m=1$, $\alpha_1:=\|s\|$ and the fixed vector $u=s$. 
Applying Proposition~\ref{prop:aligned-unifiedf} to $\Phi=\varphi$, so $w=\zeta\in\partial\varphi(g(q))$ and $\zeta\ge0$, we obtain
\[
\Big(\zeta\,\alpha_1\Big)\frac{u}{\|u\|}\ =\ \zeta\,\|s\|\frac{s}{\|s\|}\ =\ \zeta\,s\ \in\ \partial^{\,b}f(q), 
\]
which is the desired inclusion. 
\end{proof}

As a consequence, the next example collects standard applications of the chain rule in Corollary~\ref{cor:Bchain}; see \cite[Example~3.11]{Goodwin2024}.
\begin{example}\label{ex:Bchain-power-distance}
Let $\mathbb M$ be a Hadamard manifold, \(c\ge0,\ \tau\ge1,\ a\in\mathbb R\) and fix $q\in\mathbb M$. Set
\[
g(p):=d_q(p):=d(p,q),\qquad 
\varphi(t):=c\,t^{\tau}+a\quad\qquad
f:=\varphi\circ g.
\]
Then $f$ is $B$-subdifferentiable on $\mathbb M$ and \(0\in\partial^{\,b}f(q)\) and 
\[
c\tau d_q(p)^{\tau-1}\,\Big(-\frac{\log_p q}{d_q(p)}\Big)\in\partial^{\,b}f(p), \quad \quad\text{for }p\neq q.
\]
\end{example}

Next, we list standard choices of $\Phi$ and verify that each is convex and coordinatewise nondecreasing. 
These include separable convex sums built from a convex nondecreasing scalar $\psi$, and the log-sum-exp (soft-max) aggregator. 
Consequently, Proposition~\ref{prop:aligned-unifiedf} applies whenever the inputs admit a common aligned $B$-support, the composite $\Phi\!\circ (g_1,\ldots,g_m)$ is convex and its $B$-subdifferential admits an explicit description along the aligned direction.

\begin{example}\label{ex:Phi-examples-finite}
Let us give common choices of $\Phi$ satisfying the hypothesis of Proposition~\ref{prop:aligned-unifiedf}.
\begin{itemize}
\item[(i)] \emph{Pointwise maximum:} $\Phi(x):=\max_{1\le i\le m} x_i$.
The convexity of $\Phi$ follows from the fact that the pointwise maximum of affine maps is convex.
To see that $\Phi$ is monotone, note that if $x\le y$ coordinatewise, then $\max_i x_i\le \max_i y_i$.
The subdifferential of $\Phi$ at $x\in\mathbb{R}^m$ is
\[
\partial\Phi(x)=\Big\{\,w=(w_1,\ldots,w_m)\in\mathbb{R}^m:\ w_i\ge0,\ \sum_{i=1}^m w_i=1,\ \mathrm{supp}(w)\subseteq \arg\max_i x_i\,\Big\},
\]
where $\mathrm{supp}(w):=\{\,i\in\{1,\ldots,m\}:\ w_i\neq 0\,\}$.

\item[(ii)] \emph{Separable convex nondecreasing aggregator:}
Let $\psi:\mathbb R\to\mathbb R$ be convex and nondecreasing, and set
$\Phi(x):=\sum_{j=1}^m \psi(x_j)$.
The convexity of $\Phi$ follows from the convexity of each term in the sum.
For monotonicity, if $x\le y$ coordinatewise, then $\psi(x_j)\le\psi(y_j)$ for every $j$, hence $\Phi(x)\le\Phi(y)$.
The subdifferential at $x\in\mathbb R^m$ is
\[
\partial\Phi(x)=\Big\{\,\theta=(\theta_1,\ldots,\theta_m):\ \theta_j\in\partial\psi(x_j)\ \text{for all } j\,\Big\},
\]
and, because $\psi$ is nondecreasing, each $\theta_j\ge 0$, so every $w\in\partial\Phi(x)$ is coordinatewise nonnegative.

\item[(iii)] \emph{Log-sum-exp (soft-max):} for $\tau>0$ set
\(
\Phi_\tau(x):=\tau\log\!\big(\sum_{i=1}^m e^{x_i/\tau}\big).
\)
The convexity of $\Phi_\tau$ follows since $x\mapsto \sum_i e^{x_i/\tau}$ is convex and positive, and $\log$ is increasing and concave; thus the composition $\tau\log(\cdot)$ with a sum of exponentials is convex.
For monotonicity, if $x\le y$ coordinatewise then $e^{x_i/\tau}\le e^{y_i/\tau}$ for each $i$, hence $\Phi_\tau(x)\le \Phi_\tau(y)$.
The gradient, hence a subgradient, at $x\in\mathbb R^m$ is
\[
\Phi'_\tau(x)=\Big(\frac{e^{x_i/\tau}}{\sum_{j=1}^m e^{x_j/\tau}}\Big)_{i=1}^m\in \mathbb{R}^m,
\]
which is coordinatewise nonnegative and sums to $1$. Thus,  $\Phi'
_\tau(x)\in\partial\Phi_\tau(x)$.

\item[(iv)] \emph{$\ell_r$-norm on the nonnegative orthant:}
For $r\in[1,\infty)$ and $\Omega:=\{x\in\mathbb{R}^m:\ x_i\ge 0\}$, set
\[
\Phi_r(x):=\Big(\sum_{i=1}^m x_i^r\Big)^{1/r}\quad\text{for }x\in\Omega.
\]
The convexity of $\Phi_r$ holds for all $r\ge 1$ (it is the $\ell_r$-norm); restricting to $\Omega$ preserves convexity.
For monotonicity on $\Omega$, if $x\le y$ coordinatewise then $x_i^r\le y_i^r$, hence $\Phi_r(x)\le \Phi_r(y)$.
For $r>1$ and $x\neq 0$,
\[
\Phi'_r(x)
=\Big(\frac{x_i^{\,r-1}}{\big(\sum_{j=1}^m x_j^r\big)^{\frac{r-1}{r}}}\Big)_{i=1}^m \in\mathbb{R}^m.
\]
Thus,  $\Phi'_r(x)\in\partial\Phi_r(x)$ and is coordinatewise nonnegative on $\Omega$.
\end{itemize}
\end{example}

In Example~\ref{ex:Phi-examples-finite} we listed choices of $\Phi$ that satisfy the conditions of Proposition~\ref{prop:aligned-unifiedf}. Combining this with Example~\ref{ex:common-scalar-backbone} yields a Busemann-driven class covering logistic-regression/soft-max losses~\cite{collins2002logistic}. Because Busemann functions act, in a suitable sense, like linear forms on Hadamard manifolds~\cite{ghadimi2021hyperbolic}, log-sum-exp compositions of Busemann functions, motivated by information-theoretic and numerical aspects of log-sum-exp~\cite{Blanchard2021,Nielsen2016} and by hyperbolic prototype learning~\cite{keller2020theory}, provide smooth convex surrogates of max-type models that remain $B$-subdifferentiable and fit the alignment framework of Proposition~\ref{prop:aligned-unifiedf}.

\begin{example}\label{ex:aligned-global}
Fix $q_0\in\mathbb{M}$ and a nonzero $v_0\in T_{q_0}\mathbb{M}$, and let $B_{q_0,v_0}$ denote the associated Busemann function. For $j=1,\ldots,m$ define
\[
g_j(p)\ :=\ w_j\,\psi_j\!\big(B_{q_0,v_0}(p)\big)\ +\ a_j,
\]
where $w_j\ge0$, $a_j\in\mathbb{R}$, and each $\psi_j:\mathbb{R}\to\mathbb{R}$ is convex and nondecreasing. Since $B_{q_0,v_0}$ is $B$-subdifferentiable, Example~\ref{ex:common-scalar-backbone} ensures that $\{g_1,\dots,g_m\}$ belongs to the aligned $B$-support class at every $q\in\mathbb{M}$ with common direction
\(
u(q)\ :=\ \mathrm{grad}\,B_{q_0,v_0}(q)\ \in T_q\mathbb{M},
\)
and coefficients
\(
\alpha_j(q)\ \in\ w_j\,\partial\psi_j\!\big(B_{q_0,v_0}(q)\big)\ \subset [0,\infty).
\)
Now take any $\Phi$ from Example~\ref{ex:Phi-examples-finite} and set
\[
f(p)\ :=\ \Phi\big(g_1(p),\ldots,g_m(p)\big),\qquad p\in\mathbb{M}.
\]
By Proposition~\ref{prop:aligned-unifiedf}, the function $f$ is $B$-subdifferentiable at every $q\in\mathbb{M}$ and, for every  vector $(w_1,\ldots,w_m)\in\partial\Phi\big(g_1(q),\ldots,g_m(q)\big)$, one has
\[
\Big(\sum_{j=1}^m w_j\,\alpha_j(q)\Big)\,\frac{u(q)}{\|u(q)\|}\ \in\ \partial^{b} f(q).
\]
\end{example}

In the following, we state a max rule saying that the pointwise maximum of a geodesically convex and B-subdifferentiable family, continuous in the index, is B-subdifferentiable. At each point, the B-subdifferential contains those of the active indices.  This corresponds to \cite[Remark~3.9]{Goodwin2024}.

\begin{proposition}\label{prop:Bsub-max-compute}
Let $\mathbb{M}$ be a Hadamard manifold with the geodesic extension property, 
and let $A$ be compact and nonempty. 
For each $a\in A$, let $g_a:\mathbb{M}\to\mathbb{R}$ be geodesically convex.  Assume in addition that the map $A\times\mathbb{M}\ni  (a,p)\mapsto g_a(p)$ is continuous. Define $ f:\mathbb{M}\to\mathbb{R}$ by 
\begin{equation} \label{def;gmaxf}
   f(p) := \max\{ g_{a}(p):~ {a\in A}\} \qquad p\in \mathbb{M}.
\end{equation} 
Moreover, assume that each $g_a$ is B-subdifferentiable at every point. Then $f$ is B-subdifferentiable at every $q\in \mathbb{M}$. In addition, for $q\in\mathbb{M}$, letting  $A(q)=  \arg\max_{a\in A} g_a(q)$, there holds 
\begin{equation}\label{eq:Bsub-max-lower}
   \partial^{\,b} f(q)\ \supseteq\ \bigcup_{a\in A(q)} \partial^{\,b} g_a(q).
\end{equation}
\end{proposition}

\begin{remark}
Geodesic convexity of the components $g_a$ in \eqref{def;gmaxf} does not, in general, ensure B-subdifferentiability of $f$. The B-support inequality is strictly stronger than the usual convex subgradient inequality. Thus, the assumption that each $g_a$ is B-subdifferentiable is the natural, and essentially minimal, transferable condition for the maximum.
\end{remark}

We now illustrate Proposition~\ref{prop:Bsub-max-compute} with minimum–enclosing–ball models and an additively weighted cover; for more details, see \cite{Bacak2013}.

\begin{example}
Let $\mathbb{M}$ be a Hadamard manifold and let $A\subset\mathbb{M}$ be finite. For a convex nondecreasing $\varphi:[0,\infty)\to\mathbb{R}$ define
\[
f_{\varphi}(p):= \max_{a\in A} \varphi \big(d(p,a)\big),\qquad p\in\mathbb{M}.
\]
This fits Proposition~\ref{prop:Bsub-max-compute} with $g_a(p)=\varphi(d(p,a))$, hence $f_{\varphi}$ is B-subdifferentiable and its B-subdifferential is determined by the active indices.

\begin{description}
\item[Linear radius]
With $\varphi(t)=t$ we obtain $f_{\varphi}(p)=\max_{a\in A} d(p,a)$. Any minimizer $p^\ast$ is the center of the smallest closed geodesic ball containing $A$ (the circumcenter), and it is unique; see \cite[Proposition~II.2.7]{BridsonHaefliger1999}. The value $f_{\varphi}(p^\ast)$ is the circumradius; see \cite{arnaudon2013} for modeling uses.

\item[Squared radius]
With $\varphi(t)=t^2$ we obtain $f_{\varphi}(p)=\max_{a\in A} d(p,a)^2$. The function $f_g$ is strongly convex, so the minimizer is unique and equals the center of the minimal enclosing geodesic ball, with $f_{\varphi}(p^\ast)$ equal to the squared circumradius; cf. \cite[Example~2.2.18]{Bacak2014}.

\item[Additively weighted cover]
Let now $A\subset\mathbb{M}$ be compact and let $\rho:A\to\mathbb{R}_+$ be continuous. Define
\( f_\rho(p):= \max_{a\in A}\,\big(d(p,a)+\rho(a)\big),\qquad p\in\mathbb{M} \).
The minimizer $p^\ast$ is the center of the smallest geodesic ball covering
\(A_\rho\ =\ \bigcup_{a\in A} B_{\rho(a)}(a)\).
Indeed, for $r\ge0$ the inclusion $A_\rho\subset B(p,r)$ holds if and only if $d(p,a)+\rho(a)\le r$ for all $a\in A$, that is, iff $f_\rho(p)\le r$. This is a particular case of Proposition~\ref{prop:Bsub-max-compute} with $g_a(p)=\varphi_a(d(p,a))$ and $\varphi_a(t)=t+\rho(a)$, which is convex and nondecreasing in $t$. Hence $f_\rho$ is B-subdifferentiable on $\mathbb{M}$.
\end{description}
\end{example}


For more examples and properties of the Busemann subdifferential, see \cite{Goodwin2024}  and \cite{Criscitiello2025}. Definition~\ref{def:Bsub} extends the Euclidean subdifferential in a way that respects the manifold geometry. However, it is not closed under addition, i.e.,  the sum of two functions that are $B$-subdifferentiable at a point may fail to be $B$-subdifferentiable there. In particular, on general  Hadamard manifolds the $B$-subdifferential of a sum can be empty even when each term is convex and continuously differentiable. Thus the Euclidean rule $\partial^{b}h(q)=\{h'(q)\}$ need not hold; see Examples~\ref{ex:empty-Bsub-kappa} and~\ref{ex:sum-two-dist-fails} below.

\begin{example}\label{ex:empty-Bsub-kappa}
Using the notation of Example~\ref{ex:hsfbc}, fix $q\in{\mathbb H}^{n}_{1}$ and choose unit
vectors $v_{1},v_{2}\in T_{q}{\mathbb H}^{n}_{1}$ with $v_{1}\neq\pm v_{2}$. Set
$\alpha:=\langle v_{1},v_{2}\rangle$ with $-1<\alpha<1$. Define the geodesically convex
$C^{1}$ function $h:\,{\mathbb H}^{n}_{1}\to{\mathbb R}$  given by the sum of two Busemann functions
\( h(p)=B_{q,v_{1}}(p)+B_{q,v_{2}}(p),\quad p\in{\mathbb H}^{n}_{1} \).
Then $\partial^{b}h(q)=\varnothing$. Indeed, for any unit $u\in T_q{\mathbb H}^{n}_{1}$ consider the following notation for  the geodesic
\[
\gamma_u(t)=\exp_q(tu)=\cosh({t})\,q+\sinh({t})\,{u}, \qquad t\in{\mathbb R},
\]
By \eqref{eq:BFkSF}, we have 
\(
B_{q,u}\bigl(\gamma_u(t)\bigr) = - t.
\)
Since \(h\bigl(\gamma_{v_{1}}(t)\bigr)= B_{q,v_{1}}\bigl(\gamma_{v_{1}}(t)\bigr)+ B_{q,v_{2}}\bigl(\gamma_{v_{1}}(t)\bigr)\),
\begin{equation}\label{eq:h-gamma-n1}
h\bigl(\gamma_{v_{1}}(t)\bigr) = -t +  \ln\!\bigl(\cosh t - \alpha\,\sinh t\bigr)
= \ln\!\bigl(\tfrac{1-\alpha}{2} +  \tfrac{1+\alpha}{2}\,e^{-2t}\bigr)
\leq 0 ,
\end{equation}
for every $t>0$ because $\alpha \in [-1, 1]$. Similarly, we obtain that 
\begin{equation}\label{eq:h-gamma-n2}
h\bigl(\gamma_{v_{2}}(t)\bigr)= \ln\!\bigl(\tfrac{1-\alpha}{2} +  \tfrac{1+\alpha}{2}\,e^{-2t}\bigr)\leq 0.
\end{equation}
Assume, for contradiction, that $s\in\partial^{b}h(q)$.  First, consider the case $s=0$. Then, in this case, we have $h(x)\ge h(q)$,  for all $x$. But $h(q)=0$ and \eqref{eq:h-gamma-n1} gives
$h\bigl(\gamma_{v_{1}}(t)\bigr)<0$ for $t>0$, a contradiction.

Now, we assume that  $s\neq 0$.
Let $u:=s/\|s\|$ and set $\beta_1:=\langle v_{1},u\rangle \in[-1,1]$.
Using \eqref{eq:BFkSF} along $\gamma_{v_{1}}$ for $t\ge 0$,
\begin{equation}\label{eq:Bminus-s-gamma-n1}
B_{q,-s}\bigl(\gamma_{v_{1}}(t)\bigr)
=  
\ln\!\bigl(\cosh({t})+\beta_1\,\sinh({t})\bigr)
= t+ 
\ln\!\bigl(\tfrac{1+\beta_1}{2}+\tfrac{1-\beta_1}{2}e^{-2{t}}\bigr).
\end{equation}
If $\beta_1>-1$, then for all $t\ge 0$ we have 
\(
\|s\|\,B_{q,-s}\bigl(\gamma_{v_{1}}(t)\bigr)
\geq  \|s\|\,t+\ \|s\|\ln\!\bigl(\tfrac{1+\beta_1}{2}\bigr)
\).
Combining with \eqref{eq:h-gamma-n1}, the defining inequality
$h(\gamma_{v_{1}}(t))\ge h(q)+\|s\|\,B_{q,-s}(\gamma_{v_{1}}(t))$ fails for all sufficiently large $t$, a contradiction.

Now, we consider the case $\beta_1=-1$ (i.e., $u=-v_1$). To obtain a contradiction, consider instead the ray \(\gamma_{v_2}\).
Let
\[
\beta_2=  \langle v_2,u\rangle=\langle v_2,-v_1\rangle=-\alpha\ \in(-1,1).
\]
Using the similar formula as in \eqref{eq:Bminus-s-gamma-n1}, for \(t\ge0\), we obtain that 
\[
\|s\|B_{q,-s}\bigl(\gamma_{v_2}(t)\bigr)= \|s\|t+\|s\|\ln\!\bigl(\tfrac{1+\beta_2}{2}+\tfrac{1-\beta_2}{2}e^{-2t}\bigr) \geq t\|s\|+\|s\|\ln\!\bigl(\tfrac{1+\beta_2}{2}\bigr),
\]
Combining with \eqref{eq:h-gamma-n2}, the defining inequality $h(\gamma_{v_{2}}(t))\ge h(q)+\|s\|\,B_{q,-s}(\gamma_{v_{2}}(t))$ also fails for all sufficiently large $t$, a contradiction. This concludes the statement that $\partial^{b}h(q)=\varnothing$. 
\end{example}

On a Hadamard manifold, each distance map $p\mapsto d(p,x)$ is geodesically convex and $B$-subdifferentiable for every $q\neq x$. Example~\ref{ex:empty-Bsub-kappa} shows that a sum of two Busemann functions in non-opposite directions can have empty $B$-subdifferential at $q$; the next example shows the same for sums of distances, each summand has a $B$-subgradient at $q$, yet the sum may have none, so the sum rule fails.
\begin{example}\label{ex:sum-two-dist-fails}
Using the notation of Example~\ref{ex:hsfbc}, fix $q\in{\mathbb H}^{n}_{1}$ and choose unit
vectors $v_{1},v_{2}\in T_{q}{\mathbb H}^{n}_{1}$ with $v_{1}\neq\pm v_{2}$ and set $\alpha:=\langle v_{1},v_{2}\rangle\in(-1,1)$. Let $\theta:=\arccos(\alpha)\in(0,\pi)$.
For each $t>0$ and \( i=1, 2\) define
\[
a_i(t):=\exp_q(tv_{i}),\qquad \qquad
f_t(p):=d(p, a_1(t))+d(p, a_2(t)),\quad p\in{\mathbb H}^{n}_{1}.
\]
Then there exists $T>0$ such that for every $t\ge T$ one has $\partial^{b}f_t(q)=\varnothing$, whereas
\[
\partial^{b}d(\cdot, a_1(t))(q)+\partial^{b}d(\cdot, a_2(t))(q)=\{-(v_1+v_2)\}\neq\varnothing .
\]
Indeed, first note that $\partial^{b}d(\cdot,a_1(t))(q)=\{-v_1\}$, for $q \neq a_1(t)$ and $\partial^{b}d(\cdot,a_2(t))(q)=\{-v_2\}$, for $q \neq a_1(t)$, hence
\[
\partial^{b}d(\cdot,a_1(t))(q)+\partial^{b}d(\cdot,a_2(t))(q)=\{-(v_{1}+v_{2})\}\neq\varnothing, \qquad \forall t>0.
\]
To simplify the notation, define for $i\in\{1,2\}$ and $t>0$ the auxiliary functions
\begin{equation}\label{eq:defphg}
\phi_{i,t}(p):=d\bigl(p,a_i(t)\bigr)-t,\qquad \qquad  g_t(p):=\phi_{1,t}(p)+\phi_{2,t}(p)=f_t(p)-2t .
\end{equation}
Evaluating at $p=a_i(t)$ gives
\begin{equation}\label{eq:fiff-t}
\phi_{i,t}\bigl(a_i(t)\bigr)=|t-t|-t=-t .
\end{equation}
Using Lemma~\ref{le:CosLaw1}, for $j\neq i$, the hyperbolic law of cosines  with the triangle having
vertices $q$, $a_i(t)$, $a_j(t)$ yields
\begin{equation}\label{eq:cosh-tt}
\cosh d\bigl(a_i(t),a_j(t)\bigr)
=\cosh^2 t-\alpha\,\sinh^2 t
=\tfrac12\bigl(A(t)e^{t}+B(t)e^{-t}\bigr),
\end{equation}
where $A(t):=\cosh t-\alpha\sinh t$ and $B(t):=\cosh t+\alpha\sinh t$.
Since $d(\cdot,\cdot)\ge0$ we have $\cosh d(\cdot,\cdot) \ge 1$, using the elementary estimate $\operatorname{arcosh} y\le \ln(2y)$ for $y\ge 1$,  it follows from  \eqref{eq:defphg}–\eqref{eq:cosh-tt}  that 
\begin{equation}\label{eq:siff-t}
\phi_{j,t}\bigl(a_i(t)\bigr)
=\operatorname{arcosh}\!\Bigl(\tfrac12\bigl(A(t)e^{t}+B(t)e^{-t}\bigr)\Bigr)-t
\ \le\ \ln\!\bigl(A(t)+B(t)e^{-2t}\bigr)-t .
\end{equation}
Combining \eqref{eq:fiff-t} and \eqref{eq:siff-t}, we get the exact upper bound
\begin{equation}\label{eq:Rt-exact}
g_t\bigl(a_i(t)\bigr)\ \le\ \ln R_t,\qquad R_t:=\bigl(A(t)+B(t)e^{-2t}\bigr)e^{-t}
=\tfrac{1-\alpha}{2}+(1+\alpha)e^{-2t}+\tfrac{1-\alpha}{2}\,e^{-4t}.
\end{equation}
Since $\alpha\in(-1,1)$, we have $\lim_{t\to\infty}R_t=\tfrac{1-\alpha}{2}<1$; hence there exists $T_1>0$
such that $R_t<1$, for all $t\ge T_1$. Therefore, using \eqref{eq:Rt-exact} we conclude that 
\begin{equation}\label{eq:neg-on-ray}
g_t\bigl(a_i(t)\bigr)\leq \ln(\tfrac{1-\alpha}{2})<\ 0,
\qquad \forall\, t\ge T_1,\ i\in\{1,2\}.
\end{equation}
Next, for any unit $u\in T_q{\mathbb H}^{n}_{1}$ set $\beta_i(u):=\langle u,v_i\rangle$ for $i=1,2$
and $\beta^*(u):=\max\{\beta_1(u),\beta_2(u)\}$. We claim
\begin{equation}\label{eq:betastar-lb}
\beta^*(u)\geq  -\cos(\tfrac{\theta}{2})\ =:\ \beta_* \ >\ -1 .
\end{equation}
Indeed, $\|v_1+v_2\|^2=2(1+\alpha)=4\cos^2(\theta/2)$, so $\|v_1+v_2\|=2\cos(\theta/2)$, where $\theta$ is the smallest angle between $v_1$ and $v_2$.
If $\beta_1(u)<-\cos(\theta/2)$ and $\beta_2(u)<-\cos(\theta/2)$, then
$\langle u,v_1+v_2\rangle<-2\cos(\theta/2)$, contradicting Cauchy–Schwarz:
$\langle u,v_1+v_2\rangle\ge-\|v_1+v_2\|=-2\cos(\theta/2)$. Hence,  equality in
\eqref{eq:betastar-lb} holds for $u_*=-(v_1+v_2)/\|v_1+v_2\|$.

Using the explicit Busemann formula along rays  \eqref{eq:BFkSF}, we have 
\[
B_{q,-u}\bigl(\gamma_{v_i}(t)\bigr)=\ln\!\bigl(\cosh t+\beta_i(u)\sinh t\bigr)
\geq  t+\ln\!\Bigl(\tfrac{1+\beta_i(u)}{2}\Bigr).
\]
For each $u$, choose $i=i(u)\in\{1,2\}$ attaining $\beta^*(u)$. By \eqref{eq:betastar-lb}, there is a
constant $c_*:=\ln\!\bigl(\tfrac{1+\beta_*}{2}\bigr)\in\mathbb{R}$ such that
\begin{equation*}
B_{q,-u}\bigl(\gamma_{v_{i(u)}}(t)\bigr)\geq  t+c_* .
\end{equation*}
Choose $T_2$ so large that $t+c_*>0$ for all $t\ge T_2$. Finally, suppose $s\in\partial^{b}f_t(q)$ with $t\ge T:=\max\{T_1,T_2\}$.
Let $u$ be any unit vector if $s=0$, and $u:=s/\|s\|$ otherwise. Evaluating the supporting
inequality at $x=\gamma_{v_{i(u)}}(t)$ and using $g_t(q)=0$,
\[
g_t\bigl(\gamma_{v_{i(u)}}(t)\bigr)\geq  \|s\|\,B_{q,-u}\bigl(\gamma_{v_{i(u)}}(t)\bigr)\geq  0,
\]
which contradicts \eqref{eq:neg-on-ray}. Hence $\partial^{b}f_t(q)=\varnothing$ for all $t\ge T$.
\end{example}

Although Examples~\ref{ex:empty-Bsub-kappa} and \ref{ex:sum-two-dist-fails} appear \emph{negative} at first sight, the $B$-subdifferential at $q$ is empty; they actually uncover a \emph{genuinely non-Euclidean} effect on negatively curved Hadamard manifolds. Even for geodesically convex $C^1$ summands (e.g., Busemann or distance terms), the Euclidean heuristic
\[
\partial^{b}(f+g)(q)=\partial^{b}f(q)+\partial^{b}g(q)
\]
may fail so severely that $\partial^{b}(f+g)(q)=\varnothing$ while each $\partial^{b}f(q)$ and $\partial^{b}g(q)$ is nonempty. In flat spaces, the linear structure guarantees robust sum rules for convex subdifferentials without any ``alignment'' requirement; curvature breaks precisely this feature. The examples, therefore, delineate the limits of Euclidean intuition for $B$-subdifferential calculus and, at the same time, point to the additional hypothesis that restores positive results in curved settings, namely,  the existence of a \emph{common supporting direction at $q$} in the sense of Definition~\ref{def:AlignedBClass}.

\section{Projection onto horospheres in Hadamard manifolds}   \label{eq:proj}

In this section, we present the geometric foundations underlying the projection step in the Busemann proximal point algorithm. We begin by recalling the definitions of horospheres and horoballs on Hadamard manifolds, which arise as level and sublevel sets of Busemann functions and generalize the notions of hyperplanes and halfspaces in Euclidean space. We then derive a closed-form expression for the projection onto a horosphere, an essential component of our algorithm, and highlight its correspondence with the classical Euclidean case.

\begin{definition} \label{def:hhb}
Let \( B_{q, v} \) be the Busemann function associated with \( q \in {\mathbb{M}} \) and \({v}\in T_q {\mathbb{M}} \setminus \{0\} \). For a given \( c \in \mathbb{R} \), the \emph{horosphere} centered at the asymptotic direction determined by the geodesic ray \( \gamma(t) = \exp_q(t{v}) \) is the level set
\[
\mathcal{S}_{q,{v}}(c) \coloneqq \left\{ x \in {\mathbb{M}} :~B_{q, {v}}(x) = c \right\}.
\]
The corresponding \emph{horoball} is the sublevel set
\[
{\cal H}_{q,{v}}(c) \coloneqq \left\{ x \in {\mathbb{M}} :~ B_{q, {v}}(x) \leq c \right\}.
\]
\end{definition}

The following lemma shows that a geodesically convex function with unit-norm gradient increases linearly along its gradient flow, implying alignment between the gradient vector field and the geodesic velocity.
\begin{lemma}\label{lem:grad-alignment}
Let \( F : {\mathbb{M}}\to \mathbb{R} \) be a \( C^1 \) geodesically convex function, and suppose that its gradient vector field satisfies \( \|{\rm grad} F(x)\| = 1 \) for all \(x\in {\mathbb{M}}\). Fix a point \( p \in {\mathbb{M}}\), and consider the geodesic
\begin{equation}\label{eq:gamma}
   \gamma(t) := \exp_{p}\bigl(  t\, {\rm grad} F(p)\bigr), \qquad t \in \mathbb{R}.
\end{equation}
Then, the function \( F \) increases linearly along \(\gamma\), that is, \(F(\gamma(t))= F(p) + t\), for all  \(   t \in \mathbb{R} \). As a consequence, for all \( t \in \mathbb{R} \), we have
\begin{equation}\label{eq:alignment}
   \langle {\rm grad} F(\gamma(t)),\, \gamma'(t) \rangle = 1,
\end{equation}
and hence the gradient vector field is aligned with the vector tangent of the  geodesic  in \eqref{eq:gamma} as follows \({\rm grad} F(\gamma(t)) = \gamma'(t)\), for all \( t \in \mathbb{R} \).
\end{lemma}

\begin{proof}{Proof}
Set \(\psi(t):=F(\gamma(t))\). By geodesic convexity of \(F\), the function \(\psi\) is convex on \( \mathbb{R}\). By the chain rule, we have 
\begin{equation}\label{eq:defa1}
\psi'(t)=\langle {\rm grad} F(\gamma(t)),\,\gamma'(t)\rangle,\qquad t \in \mathbb{R}.
\end{equation}
Since \(\gamma'(0)= - {\rm grad} F(p)\) and \(\|{\rm grad} F(p)\|=1\), we have \(\psi'(0)=1\). Convexity of \(\psi\) then yields
\begin{equation}\label{eq:feq1}
\psi(t)\ \ge\ \psi(0)+t\psi'(0)\ =\ F(p)+t,\qquad t \in \mathbb{R}.
\end{equation}
On the other hand, geodesics have constant speed, hence \(\|\gamma'(t)\|=\|\gamma'(0)\|=\|{\rm grad} F(p)\|=1\) for all \(t \in \mathbb{R}\). Using \(\|{\rm grad} F(\cdot)\|=1\) and Cauchy–Schwarz in \eqref{eq:defa1} gives
\[
\psi'(t)=\langle {\rm grad} F(\gamma(t)),\,\gamma'(t)\rangle \le\ 1,\qquad t \in \mathbb{R}.
\]
Integrating, we obtain \(\psi(t)=F(\gamma(t)) \leq F(p)+t\), for all  \(t \in \mathbb{R}\). Together with \eqref{eq:feq1}, this implies \(\psi(t)=F(p)+t\) for all \(t \in \mathbb{R}\), proving the linear growth. Consequently, \(\psi'(t)=1\) for all \(t \in \mathbb{R}\), and \eqref{eq:defa1} gives \eqref{eq:alignment}. Since both vectors in \eqref{eq:alignment} have unit norm, it follows that \({\rm grad} F(\gamma(t))=\gamma'(t)\) for all \(t \in \mathbb{R}\).
\end{proof}
An explicit formula for the distance to a horosphere in the hyperbolic setting appears in \cite{fan2023}. The next lemma provides a model-independent projection formula on general Hadamard manifolds, included here for completeness. The proof follows the gradient flow of the Busemann function and clarifies the link with classical Euclidean orthogonal projection, used repeatedly below.

\begin{lemma}\label{le:busfunc}
Let \( q \in {\mathbb{M}} \) and \( {v}\in T_q {\mathbb{M}} \setminus \{0\} \). Let \( p \in {\mathbb{M}} \) and assume that \(p \notin \mathcal{H}_{q,{v}}(c)\), i.e., \(B_{q, {v}}(p) - c>0\). Then the projection of \(p\) onto the horosphere 
\(
\mathcal{S}_{q,{v}}(c) 
\)
is given by
\begin{equation} \label{eq:fpohs}
\mathcal{P}_{\mathcal{S}_{q,{v}}(c)}(p) = \exp_p\left( (B_{q,{v}}(p) - c)\, {\rm grad}(B_{q,{v}})(p) \right).
\end{equation} 
As a consequence, 
\(
d\big(p,\mathcal H_{q_0,v_0}(c)\big)= d\big(p,\mathcal S_{q_0,v_0}(c)\big)= B_{q_0,v_0}(p)-c.
\)
\end{lemma}

\begin{proof}{Proof}
It follows from Lemma~7 that the Busemann function \(B_{q,{v}} \) is smooth, its gradient vector field satisfies 
\(\|{\rm grad}(B_{q,{v}})(x)\| = 1\) for all \(x \in {\mathbb{M}}\), and, in addition, is Lipschitz continuous with constant \(L=1\), i.e., it satisfies \eqref{eq:1Lip}. In particular, for any \(y \in \mathcal{S}_{q,v}(c) \) we have \(B_{q,{v}}(y) = c \). Since by the continuity of \(B_{q,{v}}(p)\) \(\mathcal{S}_{q,v}(c)\) is the boundary of \(\mathcal{H}_{q,{v}}(c)\) and \(B_{q, {v}}(p) - c>0\), inequality \eqref{eq:1Lip} yields
\begin{equation}\label{eq:lipschitzm1}
 \beta:=B_{q,{v}}(p) - c \le  \min_{y \in  \mathcal{H}_{q,{v}}(c)} d(p, y)=:d(p, \mathcal{S}_{q,{v}}(c)).
\end{equation}
Define the geodesic \( \gamma : \mathbb{R} \to {\mathbb{M}} \) by
\(
\gamma(t) := \exp_p\left( t\, {\rm grad} B_{q,{v}}(p) \right). 
\)
Since \(\|{\rm grad}B_{q,v}(x)\| = 1\) for all \(x \in {\mathbb{M}}\), applying Lemma~\ref{lem:grad-alignment} with \(F:= B_{q,{v}}\) we obtain the identity
\(
B_{q,{v}}(\gamma(t)) = B_{q,{v}}(p) + t.
\)
Setting \( t = -\beta \) gives
\[
B_{q,{v}}(-\gamma(\beta)) = B_{q,{v}}(p) - \beta = c,
\]
thus \( \gamma(-\beta) \in \mathcal{S}_{q,{v}}(c) \), and, taking into account that \(\|{\rm grad}(B_{q,{v}})(x)\| = 1\), the geodesic segment from \( p \) to \( \gamma(-\beta) \) has length \( \beta \). Therefore,
\(
d(p, \mathcal{S}_{q,{v}}(c)) \le \beta.
\)
Combining \eqref{eq:lipschitzm1} with the last inequality gives \(\beta \le d(p, \mathcal{S}_{q,{v}}(c)) \le \beta\), which implies that 
\(
 d(p, \mathcal{S}_{q,v}(c)) = \beta.
\)
Thus, the point \( \gamma(-\beta) \) realizes the minimum geodesic distance from \( p \) to the horosphere. Hence, the projection is explicitly given by
\[
\mathcal{P}_{\mathcal{S}_{q,{v}}(c)}(p) = \gamma(-\beta) = \exp_p\left( (B_{q,{v}}(p) - c)\, {\rm grad}B_{q,{v}}(p) \right).
\]
This is the unique minimizer of the distance to the convex set \( \mathcal{S}_{q,v}(c) \), since Hadamard manifolds have strictly convex distance functions and horospheres are geodesically convex as level sets of convex functions. 

We proceed to prove the last equalities. Since \(B_{q, {v}}(p) - c>0\) and  \(\|{\rm grad}(B_{q,v})(x)\| = 1\), using \eqref{eq:fpohs} we obtain 
\(
d\big(p,\mathcal S_{q,v}(c)\big)=B_{q,v}(p)-c.
\)
Since \(\mathcal S_{q_0,v_0}(c)\subset\mathcal H_{q_0,v_0}(c)\) and, for \(B_{q,v}(p)>c\), the closest point in the horoball lies on its boundary, we obtain
\(
d\big(p,\mathcal H_{q_0,v_0}(c)\big)= d\big(p,\mathcal S_{q_0,v_0}(c)\big).
\)
This completes the proof.
\end{proof}

Next, we show that Lemma~\ref{le:busfunc} aligns perfectly with the classical Euclidean picture and, in $\mathbb{R}^n$, reduces to the standard orthogonal projection onto an affine hyperplane.

\begin{example}
Let the Hadamard manifold \( {\mathbb{M}} \) be the Euclidean space \( \mathbb{R}^n \) with the standard inner product. Then, \(\exp_p(w)=p+w\), \(\log_p(q)=q-p,\) and \(d(p,q)=\|p-q\|\). In addition, by using  Example~\ref{ex:defBFNDefK0}, we conclude that 
\begin{equation}\label{eq:bitsg}
B_{q,v}(p) = -\frac{1}{\|v\|}\,\langle v,\,p-q\rangle,
\qquad \quad 
{\rm grad}(B_{q,v})(p) = -\,\frac{v}{\|v\|}.
\end{equation}
Using Definition~\ref{def:hhb} together  with \eqref{eq:bitsg}, the horosphere of level \(c\in\mathbb{R}\) is the Euclidean hyperplane orthogonal to \(v\) is  as follows 
\begin{equation}\label{eq:hse}
\mathcal{S}_{q,v}(c)=\{x\in\mathbb{R}^n:~ B_{q,v}(x)=c\} =\big\{x\in\mathbb{R}^n:~ \langle v,\,x-q\rangle = -\,c\,\|v\|\big\}.
\end{equation}
Applying formula \eqref{eq:fpohs} in Lemma~\ref{le:busfunc} with \eqref{eq:bitsg}, the projection of \(p\) onto \(\mathcal{S}_{q,v}(c)\) is given by 
\begin{align}\label{eq:proj-euclidean-busemann}
\mathcal{P}_{\mathcal{S}_{q,v}(c)}(p) = p - \big(-\tfrac{1}{\|v\|}\langle v,\,p-q\rangle - c\big)\big(-\frac{v}{\|v\|}\big) = p - \frac{\langle v,\,p-q\rangle + c\|v\|}{\|v\|}\frac{v}{\|v\|}.
\end{align}
Therefore, the formula to \(\mathcal{P}_{\mathcal{S}_{q,v}(c)}(p)\) in  \eqref{eq:proj-euclidean-busemann}  coincides with the classical orthogonal projection onto the affine hyperplane \(\{x:~ \langle v,x\rangle = \langle v,q\rangle - c\|v\|\}\). Hence, in the Euclidean case, the general Riemannian projection formula of Lemma~\ref{le:busfunc} reduces to the standard orthogonal projection, confirming consistency with the familiar theory in \( \mathbb{R}^n \).
\end{example}

Now,  motivated by related constructions where horospherical fronts induce hinge-type penalties and geodesically convex decision boundaries, see, e.g., \cite{fan2023}, we construct a class of functions that satisfy Definition~\ref{def:AlignedBClass}.

\begin{example}
 Fix $q_0\in\mathbb M$ and $v_0\in T_{q_0}\mathbb M\setminus\{0\}$, and let $B_{q_0,v_0}$ be the associated Busemann function. For levels $c_1,\dots,c_m\in\mathbb R$, weights $w_i\ge0$, and constants $a_i\in\mathbb R$, define the horoballs
\[
\mathcal H_{q_0,v_0}(c_j)\ \coloneqq\ \{x\in\mathbb M:\ B_{q_0,v_0}(x)\le c_j\},\qquad j=1,\dots,m.
\]
First note that, for every $j$ and every $p\in\mathbb M$, we have 
\begin{equation}\label{eq:dist=hinge}
d\big(p,\mathcal H_{q_0,v_0}(c_j)\big)= \psi_i\!\big(B_{q_0,v_0}(p)\big),\qquad \psi_j(t):=\max\{0,t-c_j\}.
\end{equation}
Indeed, if $B_{q_0,v_0}(p)\le c_j$, then $p\in\mathcal H_{q_0,v_0}(c_j)$ and $d\big(p,\mathcal H_{q_0,v_0}(c_j)\big)=0=\max\{0,B_{q_0,v_0}(p)-c_j\}$. Assume now $B_{q_0,v_0}(p)>c_j$. Let $\mathcal S_{q_0,v_0}(c_j):=\{x\in\mathbb M:\ B_{q_0,v_0}(x)=c_j\}$ be the horosphere of level $c_i$. By Lemma~\ref{le:busfunc}, we have 
\[
d\big(p,\mathcal H_{q_0,v_0}(c_j)\big)= d\big(p,\mathcal S_{q_0,v_0}(c_j)\big)= B_{q_0,v_0}(p)-c_j=\max\{0,B_{q_0,v_0}(p)-c_j\}.
\]
This completes the proof of \eqref{eq:dist=hinge}. For each $j$, using \eqref{eq:dist=hinge} define
\begin{equation}\label{eq:dist=hingen}
g_j(p)\ \coloneqq\ w_j\,d\big(p,\mathcal H_{q_0,v_0}(c_j)\big)+a_j= w_j \psi_j\!\big(B_{q_0,v_0}(p)\big)+a_j.
\end{equation} 
Then we use the  idea of Example~\ref{ex:aligned-global} to conclude that the class $\{g_1, \ldots, g_m\}$ satisfies  Definition~\ref{def:AlignedBClass}.
\end{example}

\section{The Busemann hybrid projection-proximal point algorithm} \label{sec:proxLine}
In this section, we introduce a variant of the hybrid projection-proximal point algorithm, as proposed in \cite{SolodovBenar1999}, adapted to the setting of Hadamard manifolds for solving problem~\eqref{eq:problem}. Before presenting our main approach, let us briefly recall the classical proximal point algorithm in the Hadamard manifold context, as studied in \cite{Ferreira2002}, which is defined as follows:
 
\begin{algorithm}[H]
\begin{footnotesize}
\begin{description}
\item[Step 0.]
Take a sequence of positive parameters \( \{ \mu_k \}_{k \in \mathbb{N}} \subset (0, +\infty) \) and an initial point \( p_0 \in \mathbb M \). Set \( k \gets 0 \).

\item[Step 1.]
Given the current iterate \( p_{k} \in \mathbb M \), compute the next iterate \( p_{k+1} \in \mathbb M \) such that
\begin{equation*}
\mu_k \log_{p_{k+1}} p_{k} \in \partial f(p_{k+1}).
\end{equation*}
\textbf{Stopping test:} if \(p_{k+1}=p_{k}\), \textbf{stop} and return \(p_{k}\). 

\item[Step 2.]
If \( p_{k+1} = p_{k} \), then \textbf{stop} and return \( p_{k} \); otherwise, set \( k \gets k+1 \) and go to \textbf{Step~1}.
\end{description}
\caption{Exact proximal point algorithm}
\label{Alg:ExactPPM}
\end{footnotesize}
\end{algorithm}

Next, we introduce the {\it Busemann hybrid proximal point algorithm on Hadamard manifolds} for solving problem~\eqref{eq:problem}. This algorithm combines elements of the classical proximal point framework with the geometry of horospheres induced by Busemann functions, thereby allowing iterates to evolve along geodesically meaningful directions. The algorithm is formally described below:

\begin{algorithm}[H]
\begin{footnotesize}
\begin{description}
  \item[Step 0.]
        Choose \(p_{0}\in{\mathbb{M}}\) and two parameters  \(\sigma, \rho\in[0,1)\);
        set the iteration counter \({k}\gets 0\).

  \item[Step 1.]
        Take  \(\mu_{k}>0\) and compute a triple \(\bigl(q_{k},v_{k},\varepsilon_{k}\bigr)\) with \(q_{k}\in{\mathbb{M}}\)  and  \( v_{k},\varepsilon_{k} \in T_{q_k} {\mathbb{M}}\) satisfying the conditions
        \begin{equation} \label{eq:step1}
            0 = v_{k}-\mu_{k}\log_{q_{k}}p_{k}+\varepsilon_{k},\quad \quad  v_{k}\in \partial^b f(q_{k}), 
        \end{equation} 
      with  the relative error  \(\varepsilon_{k}\)  satisfying 
        \begin{equation} \label{eq:step1Error}   
        \|\varepsilon_{k}\|  \leq   \sigma\max\left\{\mu_k d(q_k,p_k),~ \|v_k\|\right\},
          \end{equation} 
        
        \textbf{Stopping test:} if \(v_{k}=0\) or \(q_{k}=p_{k}\), then \textbf{stop} and return \(p_{k}\). 
\item[Step 2.]      
        Define the horosphere
         \begin{equation*} 
            \mathcal S_{q_k, -v_k}=
            \bigl\{p\in{\mathbb{M}} :~ B_{q_{k},-v_{k}}(p)=0\bigr\},
          \end{equation*} 
        and set the next iterate by projecting \(p_{k}\) onto \(\mathcal S_{q_k, v_k}\):
         \begin{equation*} 
            p_{k+1}:={\cal P}_{\mathcal S_{q_k, -v_k}}(p_{k}).
           \end{equation*}

  \item[Step 3.]
        Set \({k}\gets {k+1}\) and go to \textbf{Step~1}.
\end{description}
\caption{Busemann hybrid projection-proximal point algorithm (BHPPM)}
\label{Alg:HHPPM}
\end{footnotesize}
\end{algorithm}
We begin by discussing the stopping criterion and its implications for Algorithm~\ref{Alg:HHPPM}. Suppose the algorithm terminates at some iteration \(k \in \mathbb{N}\). Two cases may arise: either \(v_k = 0\) or \(q_k = p_{k}\). In the first case, the inclusion in~\eqref{eq:step1} reduces to \(0 \in \partial^{b} f(q_k)\), implying by  Proposition~\ref{pr:ocp} that   \(q_k\) is a solution to problem~\eqref{eq:problem}. In the second case, \(q_k = p_{k}\) implies \(v_k = -\varepsilon_k\) from~\eqref{eq:step1}. Since the error satisfies \(\|\varepsilon_k\| \le \sigma \|v_k\|\) with \(0 \le \sigma < 1\), it follows that \(v_k = 0\), and we again conclude that \(q_k\) solves the problem. Thus, if the algorithm terminates,  it returns a solution to problem~\eqref{eq:problem}. Based on this discussion, we {\it assume without loss of generality that the algorithm generates an infinite sequence \(\{p_{k}\}_{k \in \mathbb{N}}\) throughout the remainder of the paper}. We also assume throughout the analysis that the solution set of problem~\eqref{eq:problem} is nonempty, i.e.,  
\[
S^*:=\arg\min_{p\in\mathbb{M}}f(p)\neq \varnothing.
\]
We also note that  BHPPM merges the proximal point framework with geometric projections onto horospheres defined by Busemann functions. At each iteration, a subgradient condition involving a proximity term and an error tolerance is imposed, with inexactness controlled by a parameter \(\sigma \in [0,1)\). The key feature of BHPPM lies in its update rule; instead of using standard proximal mappings, the current iterate is projected onto a horosphere centered at a subgradient direction. This approach incorporates the intrinsic geometry of the manifold into the iteration process. When \(\sigma = 0\), the algorithm becomes exact and coincides with the classical proximal point algorithm on Hadamard manifolds, as will be shown below.  The main idea behind the algorithm is that the horosphere used in the projection step separates the current iterate from the solution set, guiding the sequence toward optimality. To illustrate this, assume that \({p^*}\in S^*\) and that \(\sigma = 0\). In this case, since \(v_k \in \partial^b f(q_k)\), it follows from Definition~\ref{def:Bsub} that
\[
f({p^*}) \geq f(q_k) + \|v_k\| B_{q_k, -v_k}({p^*}).
\]
Since the sequence \(\{p_{k}\}_{k \in \mathbb{N}}\) is infinite, we must have \(v_k \neq 0\), and therefore we conclude that
\[
B_{q_k, -v_k}({p^*}) \leq 0.
\]
On the other hand, since \(\sigma = 0\), the error term \(\varepsilon_k\) vanishes, and from~\eqref{eq:step1} it follows that
\[
v_k := \mu_k \log_{q_k} p_k \in \partial^b f(q_k),
\]
which implies that \(p_k = \exp_{q_k}(v_k / \mu_k)\). Thus,  taking into account  that \( v_k\neq 0\) and considering that $\mu_{k}>0$,   by applying Lemma~\ref{eq:pbfu} we obtain
\[
B_{q_k, -v_k}(p_k)=B_{q_k, -v_k}(\exp_{q_{k}}(v_{k}/\mu_{k})) =  \|v_{k}\|/\mu_{k}>0.
\]
Consequently, the horosphere \(\mathcal{S}_{q_k, -v_k}\) separates the current iterate \(p_k\) from the solution \({p^*}\). This geometric separation motivates the projection step onto the horosphere \(\mathcal{S}_{q_k, -v_k}\)  in the algorithm and guarantees descent toward the solution set.

Before proceeding with the analysis of Algorithm~\ref{Alg:HHPPM}, we show that, in the exact case (\(\sigma = 0\)), Algorithm~\ref{Alg:HHPPM} reduces to Algorithm~\ref{Alg:ExactPPM}.

\begin{remark} 
Let \(\{p_{k}\}_{k\in\mathbb{N}} \subset {\mathbb{M}}\) be the sequence generated by Algorithm~\ref{Alg:HHPPM}, and assume that \( p_{k} \neq q_k \). If the inexactness parameter is set to \(\sigma = 0\), then the algorithm satisfies
\begin{equation}\label{eq:vdef}
       v_k = \mu_k \log_{q_k} p_{k} \in \partial^{b} f(q_k), \qquad p_{k+1} = \mathcal{P}_{\mathcal{S}_{q_k, -v_k}}(p_{k}), \qquad \forall k \in \mathbb{N},
\end{equation}
where
\[
   \mathcal{S}_{q_k, -v_k} := \left\{ p \in {\mathbb{M}} :~B_{q_k, -v_k}(p) = 0 \right\}, 
\]
is the horosphere of level zero associated with the Busemann function \( B_{q_k, -v_k} \), and \( \mathcal{P}_{\mathcal{S}_{q_k, -v_k}} \) denotes the metric projection onto this horosphere. To show that Algorithm~\ref{Alg:HHPPM} reduces to Algorithm~\ref{Alg:ExactPPM} in this case, it suffices to prove that \( q_k = \mathcal{P}_{\mathcal{S}_{q_k, -v_k}}(p_{k}) \), i.e., \( p_{k+1} = q_k \). For that, define
\[
\rho := d(q_k, p_{k}) = \|\log_{q_k} p_{k}\| > 0, \qquad \hat{v}_k := \frac{v_k}{\|v_k\|}.
\]
Then, from the equality in \eqref{eq:vdef} we have
\begin{equation} \label{eq:logalign}
\log_{q_k} p_{k} =\rho \hat{v}_k, \qquad \|v_k\| = \mu_k \rho,
\end{equation}
which implies
\(
p_{k} = \exp_{q_k}(\rho \hat{v}_k),
\)
i.e., \( p_{k} \) lies along the geodesic ray \( t \mapsto \exp_{q_k}(t \hat{v}_k) \) at time \( t = \rho \). By Lemma~\ref{eq:pbfu}, this yields
\begin{equation*}
B_{q_k, -v_k}(p_{k}) = \rho.
\end{equation*}
Moreover, Lemma~\ref{le:CharactBusFunc} ensures that the Busemann function \(B_{q_k, -v_k} \) is continuously differentiable, and \( \|{\rm grad}(-B_{q_k, -v_k})(p_{k})\| = 1 \). Thus, by Lemma~\ref{le:busfunc}, the metric projection of \( p_{k} \) onto the horosphere \(\mathcal{S}_{q_k,-v_k}\) is given by
\begin{equation} \label{eq:projformula}
\mathcal{P}_{\mathcal{S}_{q_k, -v_k}}(p_{k}) = \exp_{p_{k}}\bigl( \rho \, {\rm grad}B_{q_k, -v_k}(p_{k}) \bigr).
\end{equation}
On the other hand, since \( p_{k} = \exp_{q_k}(\rho \hat{v}_k) \), we conclude from  identity \eqref{eq:gradbf2} in  Lemma~\ref{le:CharactBusFunc}  that 
\[
{\rm grad}B_{q_k, -v_k}(p_{k}) = P_{q_k p_{k}} \hat{v}_k.
\]
Combining this with \eqref{eq:logalign} and \eqref{eq:projformula}, and using that \( P_{q_k p_{k}} \log_{q_k} p_{k} = -\log_{p_{k}} q_k \),  we obtain that 
\[
\mathcal{P}_{\mathcal{S}_{q_k, -v_k}}(p_{k}) = \exp_{p_{k}}\bigl(-\rho  P_{q_k p_{k}} \hat{v}_k \bigr) = \exp_{p_{k}}\bigl(-P_{q_k p_{k}} \log_{q_k} p_{k} \bigr)=\exp_{p_{k}}( \log_{p_{k}} q_k ) = q_k, 
\]
which proves that \( p_{k+1} = q_k \). Finally, since horospheres are geodesically convex in Hadamard manifolds, the metric projection onto them is uniquely defined, see Proposition~\ref{prop:proj-ineq}. Therefore, the update in inclusion  \eqref{eq:vdef} becomes $\mu_k \log_{p_{k+1}} p_{k} \in \partial^{b} f(p_{k+1})$. Since, Proposition~\ref{pr:BSubset} implies that  $\partial^{b} f(p_{k+1})\subset \partial f(p_{k+1})$ we conclude that 
\[
\mu_k \log_{p_{k+1}} p_{k} \in  \partial  f(p_{k+1}),
\]
which matches Step~1 of Algorithm~\ref{Alg:ExactPPM}. Hence, in the exact case (\(\sigma = 0\)), Algorithm~\ref{Alg:HHPPM} reduces to Algorithm~\ref{Alg:ExactPPM}.
\end{remark}

\subsection{Convergence Analysis}

In this section, we analyze the convergence properties of the sequence \(\{p_{k}\}_{k \in \mathbb{N}}\), assuming it is infinite. We begin by proving a descent lemma that ensures the squared distance from each iterate to the solution set decreases along the sequence. This descent behavior is fundamental to our convergence analysis and can be regarded as a Riemannian counterpart of the classical estimate in linear spaces, as presented in~\cite[Lemma~2.1]{SolodovBenar1999}.

\begin{lemma}\label{le:proj-descent-hadamard}
Let \(\mathbb{M}\) be a Hadamard manifold,  \(p, q \in {\mathbb{M}} \) and \( v \in T_q {\mathbb{M}} \setminus \{0\} \). Assume that \(p^*\in S^*\) and  \(v \in \partial^b f(q)\). Let \( \mathcal{S}_{q, -v}:= \{x \in {\mathbb{M}} :~ B_{q,-v}(x)=0\}\) be a horosphere. Then, there holds
\[
d^2(\mathcal{P}_{\mathcal{S}_{q, -v}(c)}(p), {p^*}) \le d^2(p, {p}^*) - \left(B_{q,-v}(p)\right)^2, \qquad \forall p\notin \mathcal{H}_{q,v}.
\]
\end{lemma}
\begin{proof}{Proof}
Since \(v \in \partial^b f(q)\), from Definition~\ref{def:Bsub} it follows that
\(
f({p}^*) \geq f(q) + \|v\| B_{q, -v}({p}^*).
\)
Thus, considering that  \(v \neq 0\), we conclude that
\(
B_{q, -v}({p}^*)\leq 0.
\)
Hence, \({p}^*\in \mathcal{H}_{q,-v}\). On the other hand, for a given \(z\in {\mathbb{M}}\), it follows from Lemma~\ref{le:CosLawF} that 
\begin{equation} \label{eq:opop}
d^2(z, p) + d^2(z, {p}^*) - 2\langle \log_z p, \log_z {p}^* \rangle \leq d^2(p, {p}^*).
\end{equation}
Since \(\mathcal{H}_{q,-v}\) is geodesically convex, the metric projection onto \(\mathcal{H}_{q,-v}\) is well-defined and unique. Considering that the horosphere \( \mathcal{S}_{q,-v}\) is the boundary of the horoball \(\mathcal{H}_{q,v}\), for a given \(p\notin \mathcal{H}_{q,v}\), we have \(\mathcal{P}_{\mathcal{H}_{q, -v}}(p)= \mathcal{P}_{\mathcal{S}_{q, -v}}(p)\). Thus, letting \(z := \mathcal{P}_{\mathcal{S}_{q, -v}}(p)\) and taking into account that \({p}^*\in \mathcal{H}_{q,-v}\), by applying Proposition~\ref{prop:proj-ineq} we obtain  that
\(
\langle \log_z p, \log_z {p}^* \rangle \leq 0.
\)
Combining this inequality with \eqref{eq:opop}, we conclude that 
\begin{equation}\label{eq:ineq-dz}
d^2(z, p) + d^2(z, {p}^*) \leq d^2(p, {p}^*).
\end{equation}
Furthermore, Lemma~\ref{le:busfunc} shows that the projection of \(p\) onto the horosphere \(\mathcal{S}_{q,v}\) is given by
\[
z = \mathcal{P}_{\mathcal{S}_{q, -v}}(p) = \exp_p\left(B_{q,-v}(p){\rm grad}(B_{q,-v})(p) \right).
\]
Taking into account that the gradient of the Busemann function satisfies \(\|{\rm grad}(B_{q,-v})(p)\| = 1\) and \(B_{q,-v}(p)>0\), we obtain that 
\(
d(z, p) = \|B_{q,-v}(p)  {\rm grad}(B_{q,-v})(p) \| = B_{q,v}(p).
\)
Squaring both sides yields
\(
d^2(z, p) = (B_{q,v}(p))^2,
\)
and substituting into \eqref{eq:ineq-dz} completes the proof.
\end{proof}

We are now prepared to establish our \emph{main global convergence result} for Algorithm~\ref{Alg:HHPPM}. The theorem below describes the asymptotic behavior of the sequence generated by  Algorithm~\ref{Alg:HHPPM}. 
\begin{theorem} \label{eq:mainth}
Let \(\{p_{k}\}_{k\in\mathbb{N}}\) be any sequence generated by Algorithm~\ref{Alg:HHPPM} and \({p}^*\in S^*\).  Then, the following inequality  hold
\begin{equation} \label{eq:fctenpi}
d^2(p_{k+1}, {p}^*) \le d^2(p_{k}, {p}^*) - \left( \frac{1 - \sigma}{1 + \sigma} \right)^2  d^2(p_{k}, q_{k}),  \qquad \forall   k\in \mathbb{N}.
\end{equation}
As a consequence,  the sequence \(\{p_{k}\}_{k\in\mathbb{N}}\) is bounded, and the sequence \(\{d(q_{k},p_{k})\}_{k\in\mathbb{N}}\) converges  to zero.
In addition, if \( \sum_{k=0}^{+\infty} \mu_k^{-2} = +\infty \), then there exists a subsequence of \(\{v_k\}_{k\in\mathbb{N}}\) which converges  to zero, and the sequence \(\{p_{k}\}_{k\in\mathbb{N}}\)  converges to a solution of problem~\eqref{eq:problem}.
\end{theorem}

\begin{proof}{Proof}
Let \({p}^*\in S^*\) be an arbitrary solution.  Since \(v_k \in \partial^b f(q_k)\), from Definition~\ref{def:Bsub} it follows that
\(
f({p}^*) \geq f(q_k) + \|v_k\| B_{q_k, -v_k}({p}^*), 
\)
for all \(k\in \mathbb{N}\).  Since the sequence \(\{p_{k}\}_{k \in \mathbb{N}}\) is infinite, we must have \(v_k \neq 0\), and therefore we conclude that
\begin{equation} \label{eq:opcod}
B_{q_k, -v_k}({p}^*) \leq 0, \qquad \forall   k\in \mathbb{N}.
\end{equation} 
We now prove that \(B_{q_k,-v_k}(p_k)>0\). To this end, we consider the two possible cases for the relative error \(\varepsilon_k\) in~\eqref{eq:step1Error}. We begin by considering the case in which the error satisfies
\begin{equation} \label{eq:FirstError}
\|\varepsilon_k\|\leq \sigma\mu_k d(q_k,p_k).
\end{equation}
Using the equality in~\eqref{eq:step1}, we have \( v_{k} = \mu_{k} \bigl( \log_{q_{k}} p_{k} \bigr) - \varepsilon_{k} \). Hence, applying Lemma~\ref{le:Inbusfunc} and then the Cauchy--Schwarz inequality, and taking into account that \( \|\log_{q_k} p_k\| = d(q_k, p_k) \), we obtain
\begin{equation*}
   \|v_k\| B_{q_k, -v_k}(p_{k}) \geq \langle v_k, \log_{q_k} p_k \rangle \geq \mu_k d^2(q_k, p_k) - \|\varepsilon_k\| d(q_k, p_k).
\end{equation*}
Now, by using \eqref{eq:FirstError} we have  
\begin{equation} \label{eq:fctenp}
 \|v_k\|B_{q_k,-v_k}(p_{k})\geq  {\mu_k (1 - \sigma) d^2(p_{k}, q_{k})}.
\end{equation} 
On the other hand,  equality in \eqref{eq:step1} implies  that \(\mu_i d(p_{k}, q_{k}) = \|v_{k} + \varepsilon_{k}\|\geq \|v_{k} \|- \|\varepsilon_{k}\| \). Thus,  fom \eqref{eq:FirstError}  it follows   that 
\(
\mu_i d(p_{k}, q_{k})  \geq \|v_{k}\| - \sigma\mu_k d(q_k,p_k),
\)
which implies
\begin{equation} \label{eq:fctenpp}
\mu_i d(p_{k}, q_{k}) \geq \frac{1}{1 + \sigma} \|v_{k}\|.
\end{equation} 
 Given that the sequence \(\{p_k\}_{k \in \mathbb{N}}\) is infinite, it follows that \(v_k \neq 0\) and  \(q_k \neq p_k\). Hence, combining  \eqref{eq:fctenpp} with \eqref{eq:fctenp}, we obtain that 
\begin{equation} \label{eq:fctenpn1}
  B_{q_k,-v_k}(p_k)\geq  \frac{1 - \sigma}{1 + \sigma} d(p_{k}, q_{k})>0.
\end{equation}
 We proceed to analyze the second bound imposed on the relative error \(\epsilon_k\), given by
\begin{equation} \label{eq:secondError}
\|\varepsilon_k\| \leq \sigma \|v_k\|.
\end{equation}
Applying  Lemma~\ref{le:Inbusfunc} and rewriting  the inclusion in~\eqref{eq:step1} as 
\(
v_k + \varepsilon_k = \mu_k \log_{q_k} p_k.
\)
we have 
\[
  \|v_k\|B_{q_k, -v_k}(p_{k})\geq  \langle {v_k},\log_{q_k}p_k\big\rangle \geq \frac1{\mu_k}\langle v_k,v_k+\varepsilon_k\rangle.
\]
Using the estimate
\(
\langle v_k, v_k + \varepsilon_k \rangle \geq \|v_k\|^2 - \|v_k\|\|\varepsilon_k\|,
\)
along with the error bound from~\eqref{eq:secondError}, we obtain the following expression
\begin{equation}  \label{eq:fcte2}
  \|v_k\|B_{q_k, v_k}(p_{k}) \geq \frac1{\mu_k}(\|v_k\|^2- \|v_k\|\|\varepsilon_k\|)\geq\left({1-\sigma}\right)\frac{1}{\mu_k}\|v_k\|^2. 
\end{equation} 
In the second case~\eqref{eq:secondError}, and recalling that \(v_k \neq 0\), the inequality~\eqref{eq:fcte2} implies
\begin{equation} \label{eq:fctenpn}
B_{q_k,-v_k}(p_k) \geq \frac{1 - \sigma}{\mu_k} \|v_k\|.
\end{equation}
On the other hand, using again~\eqref{eq:step1} and taking into account the error bound~\eqref{eq:secondError}, we obtain
\[
\mu_k d(p_k, q_k) = \|v_k + \varepsilon_k\| \leq \|v_k\| + \|\varepsilon_k\| \leq (1 + \sigma) \|v_k\|,
\]
which in turn implies the lower bound
\[
\|v_k\| \geq \frac{\mu_k}{1 + \sigma} d(p_k, q_k).
\]
Finally, combining the last inequality with \eqref{eq:fctenpn}, we deduce that \eqref{eq:fctenpn1} also holds in this case. Since we are assuming \(p_k \neq q_k\), it follows in particular that \(p_k \notin H_{q_k, v_k}\), for all \(k \in \mathbb{N}\). On the other hand, from \eqref{eq:opcod}, we know that \(p^* \in H_{q_k, v_k}\) for all \(k \in \mathbb{N}\). Hence, considering that  \(p_{k+1} := \mathcal{P}_{\mathcal{S}_{q_k, v_k}}(p_k)\), we can invoke Lemma~\ref{le:proj-descent-hadamard} to conclude that the following inequality holds
\begin{equation} \label{eq:fcten}
d^2(p_{k+1}, {p}^*) \le d^2(p_{k}, {p}^*) - \left(B_{q_{k},-v_{k}}(p_{k}) \right)^2, \qquad \forall   k\in \mathbb{N}.
\end{equation}
Combining  \eqref{eq:fcten}  with~\eqref{eq:fctenpn1}, we obtain the desired estimate~\eqref{eq:fctenpi}. As a consequence,~\eqref{eq:fctenpi} shows that the sequence \(\{d(p_{k}, {p}^*)\}_{k \in \mathbb{N}}\) is monotonically non-increasing and therefore bounded. Thus,  the sequence \(\{p_k\}_{k \in \mathbb{N}}\) is also bounded. Applying~\eqref{eq:fctenpi} once more, we deduce that
\[
\left( \frac{1 - \sigma}{1 + \sigma} \right)^2 d^2(p_{k}, q_{k}) \le d^2(p_{k}, {p}^*) - d^2(p_{k+1}, {p}^*), \quad \forall\, k \in \mathbb{N}.
\]
Summing the above inequality from \(k = 0\) to \(k = N - 1\), we obtain
\[
\left( \frac{1 - \sigma}{1 + \sigma} \right)^2 \sum_{k=0}^{N-1} d^2(p_{k}, q_{k}) \leq d^2(p_0,{p}^*) - d^2(p_{N},{p}^*) \leq d^2(p_0,{p}^*).
\]
This estimate implies that the sequence \(\{d(q_{k}, p_{k})\}_{k \in \mathbb{N}}\) converges to zero, which completes the first part of the proof.

We now turn to the final part of the theorem. To this end, we assume that \( \sum_{k=0}^{+\infty} \mu_k^{-2} = +\infty \). We again consider the two possible cases for the relative error \(\varepsilon_k\) in~\eqref{eq:step1Error}. First, suppose that \eqref{eq:FirstError} holds. By combining inequalities \eqref{eq:fctenp} and \eqref{eq:fctenpp}, we obtain
\begin{equation} \label{eq:bbmth}
B_{q_k,-v_k}(p_{k}) \geq \frac{1 - \sigma}{\mu_k(1 + \sigma)^2} \|v_{k}\|.
\end{equation} 
Next, assume that the second condition \eqref{eq:secondError} is satisfied. Dividing the right-hand side of \eqref{eq:fctenpn} by \((1 + \sigma)^2\), we see that the same bound \eqref{eq:bbmth} holds in this case as well. Substituting this estimate into inequality \eqref{eq:fcten}, we arrive at
\begin{equation} \label{eq:fctenm}
d^2(p_{k+1}, {p}^*) \leq  d^2(p_k, {p}^*) - \left( \frac{1 - \sigma}{\mu_k (1 + \sigma)^2} \right)^2 \|v_k\|^2.
\end{equation}
To prove that the sequence \(\{v_k\}_{k\in\mathbb{N}}\) admits a subsequence converging to zero, suppose by contradiction that \( \liminf_{k \to \infty} \|v_k\| = 2\delta > 0 \). Then there exists \( k_0 \in \mathbb{N} \) such that \( \|v_k\| \ge \delta \) for all \( k \ge k_0 \). Using~\eqref{eq:fctenm}, this implies that
\[
 \delta^2 \frac{(1- \sigma)^2}{(1 + \sigma)^4} \sum_{j=k_0}^{k-1} \mu_j^{-2}= \sum_{j=k_0}^{k-1} \left( \frac{1 - \sigma}{\mu_j(1 + \sigma)^2} \right)^2 \delta^2\leq d^2(p_{k_0}, {p}^*) - d^2(p_k, {p}^*)\leq d^2(p_{k_0}, {p}^*).
\]
Since the left-hand side tends to \(+\infty\) as \(k\) goes to +\(\infty\), we reach a contradiction. Thus, we conclude that \(\liminf_{k \to \infty} \|v_k\| = 0\), which implies that the sequence \(\{v_k\}_{k \in \mathbb{N}}\) admits a subsequence converging to zero. This establishes the first part of the final statement of the theorem.

 Since the sequence  \(\{p_{k}\}_{k\in\mathbb{N}}\) is bounded, we proceed to prove that there exists an accumulation point \(\{p_{k}\}_{k\in\mathbb{N}}\) that is solution of problem~\eqref{eq:problem}.  Let \(\{v_{k_\ell}\}_{\ell \in \mathbb{N}}\) be such a subsequence with \(\lim_{\ell \to +\infty} v_{k_\ell} = 0\). Since the sequence \(\{p_k\}_{k \in \mathbb{N}}\) is bounded, we may assume, without introducing new indices, that \(\{p_{k_\ell}\}_{\ell \in \mathbb{N}}\) converges to some point \(\hat{p} \in {\mathbb{M}}\). By item~(i), we have \(\lim_{\ell \to +\infty} d(p_{k_\ell}, q_{k_\ell}) = 0\), which implies that \(\lim_{\ell \to +\infty} q_{k_\ell} = \hat{p}\) as well.  Since  \eqref{eq:step1} implies that  \(v_{k_\ell} \in \partial^b f(q_{k_\ell})\) and Proposition~\ref{pr:BSubset} implies that  \(\partial^b f(q)\subset \partial f(q)\),  we conclude that  \(v_{k_\ell} \in \partial f(q_{k_\ell})\). Thus, 
 \[
 f({p}) \geq f(q_{k_\ell})+\left\langle  v_{k_\ell}, \log_{q_{k_\ell}} p\right\rangle, \qquad \forall p\in {\mathbb{M}}.
 \]
 Letting \(\ell\)  goes to \(+\infty \), and using that   \(\lim_{\ell \to +\infty} q_{k_\ell}=\hat{p}\) and \(\lim_{\ell \to +\infty} v_{k_\ell} = 0\), we obtain
 
 \[
 f({p}) \geq f(\hat{p})+\left\langle  0, \log_{\hat{p}} p\right\rangle, \qquad \forall p\in {\mathbb{M}}, 
 \]
 which implies \( 0 \in \partial f(\hat{p}) \), that is, \( \hat{p} \) is a solution to problem~\eqref{eq:problem}. Finally, we prove the convergence of the entire sequence. It suffices to show that it has a unique accumulation point. Suppose, let  \(\tilde{p} \in {\mathbb{M}}\) be other  accumulation point of the sequence \(\{p_k\}_{k \in \mathbb{N}}\). Then, there exist two subsequences \(\{p_{k_i}\}_{i \in \mathbb{N}}\) and \(\{p_{k_j}\}_{j \in \mathbb{N}}\) such that 
\[
\lim_{i \to +\infty} p_{k_i} = \hat{p} \quad \text{and} \quad \lim_{j \to +\infty} p_{k_j} = \tilde{p}.
\]
Since equation \eqref{eq:fctenm} holds for every solution, in particular it holds for \({p}^* = \hat{p}\). Hence, it follows from \eqref{eq:fctenm} that the sequence \(\{d(p_k, \hat{p})\}_{k \in \mathbb{N}}\) is strictly decreasing and bounded below, and therefore convergent. Furthermore, the subsequence \(\{d(p_{k_i}, \hat{p})\}_{i \in \mathbb{N}}\) converges to zero, implying that the full sequence \(\{d(p_k, \hat{p})\}_{k \in \mathbb{N}}\) also converges to zero. Consequently,
\[
d(\tilde{p}, \hat{p}) = \lim_{j \to +\infty} d(p_{k_j}, \hat{p}) = 0,
\]
and thus \(\tilde{p} = \hat{p}\). Therefore, all accumulation points of the sequence \(\{p_k\}_{k \in \mathbb{N}}\) must coincide with \(\hat{p}\). By the completeness of the manifold, it follows that the entire sequence converges to \(\hat{p}\). Since \(\hat{p}\) was previously shown to be a solution of problem~\eqref{eq:problem}, the proof is complete.
\end{proof}

We now establish a nonasymptotic iteration complexity bound for the Busemann hybrid proximal point algorithm. The result below shows that the sequence generated by Algorithm~\ref{Alg:HHPPM} achieves a sublinear rate of convergence with respect to the Busemann residual and the norm of approximate subgradients.

\begin{theorem} \label{thm:BHPPM-complexity}
Let \(\{p_{k}\}_{k\in\mathbb{N}}\) be any sequence generated by Algorithm~\ref{Alg:HHPPM}  with inexactness parameter $\sigma\in[0,1)$.  Then,   for an arbitrary solution  \({p}^*\in S^*\) and  for given \(N\geq 1\) there holds
\begin{equation} \label{eq:comp1}
       \min_{0\le k\le N-1} d(p_k,q_k)\leq   \frac{{(1+\sigma})d(p_0,p^*)}{{1-\sigma}} \frac{1}{\sqrt{N}}, 
      \end{equation} 
      In addition, if  $\mu_k\geq {\bar \mu}>0$ for every $k$, then
 \begin{equation} \label{eq:comp2}
        \min_{0\le k\le N-1}\|v_k\|\leq  {\bar \mu} \frac{{(1+\sigma})d(p_0,p^*)}{{1-\sigma}} \frac{1}{\sqrt{N}}.
\end{equation}
\end{theorem}
\begin{proof}{Proof}
The proof is based on the descent estimate obtained in Theorem~\ref{eq:mainth}, namely
\begin{equation*} 
 \left( \frac{1 - \sigma}{1 + \sigma} \right)^2  d^2(p_{k}, q_{k}) \leq   d^2(p_{k}, {p}^*) - d^2(p_{k+1}, {p}^*),  \qquad \forall   k\in \mathbb{N}.
\end{equation*}
Summing the last inequality from \(k=0\) to \(k=N-1\) gives
\[
\left( \frac{1 - \sigma}{1 + \sigma} \right)^2 \sum_{k=0}^{N-1} d^2(p_k, q_k) \leq 
\sum_{k=0}^{N-1} \bigl( d^2(p_k, p^*) - d^2(p_{k+1}, p^*)\bigr).
\]
The right-hand side is a telescoping sum
\[
\sum_{k=0}^{N-1} \bigl( d^2(p_k, p^*) - d^2(p_{k+1}, p^*)\bigr)
= d^2(p_0, p^*) - d^2(p_N, p^*)\leq  d^2(p_0, p^*),
\]
since squared distances are non-negative.   Therefore, we conclude that 
\[
 \left( \frac{1 - \sigma}{1 + \sigma} \right)^2 \sum_{k=0}^{N-1} d^2(p_k, q_k)
\leq 
d^2(p_0, p^*), 
\]
and this inequality gives the bound stated in~\eqref{eq:comp1}. From the inexact optimality condition in Step~1 of Algorithm~\ref{Alg:HHPPM}, we have 
\(
\mu_k \log_{q_k}p_k + \varepsilon_k + v_k = 0,
\)
and \(v_k\in\partial^b f(q_k),\) together with the relative-error rule~\eqref{eq:step1Error}, it holds that
\(
\|\varepsilon_k\| \leq \sigma \|\mu_k\log_{q_k}p_k\|
= \sigma \mu_k d(q_k, p_k).
\)
Then,
\begin{equation*} 
\|v_k\|=\|\mu_k\log_{q_k}p_k + \varepsilon_k\| \leq   \mu_k d(q_k, p_k) + \sigma\mu_k d(q_k, p_k)=(1+\sigma)\mu_k d(p_k, q_k).
\end{equation*}
Using  that \(\mu_k\geq \bar\mu>0\) for all \(k\), it follows that \(\min_{0\le k\le N-1} \|v_k\| \leq (1+\sigma)\bar\mu \min_{0\le k\le N-1} d(p_k, q_k)\). Thus,  by using the bound obtained in~\eqref{eq:comp1} we obtain \eqref{eq:comp2} and this completes the proof.
\end{proof}

\subsection{The proximal subproblem} \label{sec:Numerics}
In this section, we present the subroutine \textsc{ApproxTriple}, shown in Algorithm~\ref{Alg:InnerSolver}, 
which serves as the computational procedure  for Step~1 of the Busemann proximal point algorithm (BHPPM)  stated in Algorithm~\ref{Alg:HHPPM} 
using the gradient or subgradient algorithm on Hadamard manifolds. 
Its purpose is to compute an approximate solution \(q_k, v_k\) and  \(\varepsilon_k\)  to the regularized subproblem
\begin{equation} \label{eq:subp}
  \min_{p\in\mathbb{M}}\;f(p)+\frac{\mu_k}{2}\,d^2(p,p_k),
\end{equation} 
subject to the inexact optimality condition
\[
  0 = v_k - \mu_k\log_{q_k}p_k + \varepsilon_k,
  \qquad v_k\in\partial^{\,b}f(q_k),
\]
together with the relative error condition
\[
\|\varepsilon_k\| \leq \sigma\max\left\{\mu_k d(q_k,p_k),\|v_k\|\right\}.
\]
To begin, observe that the objective function of \eqref{eq:subp} is strongly convex and hence coercive, see \cite{Ferreira2002}. It follows that \eqref{eq:subp} admits a unique (exact) solution and, therefore, by Proposition~\ref{pr:ocp}, the objective function  is $B$-subdifferentiable at that solution. In practice, we solve \eqref{eq:subp} inexactly to balance numerical effort and convergence guarantees, allowing early termination of the inner iterations without degrading the overall convergence rate of BHPPM. Routine strategies based on the Riemannian subgradient method introduced in \cite{ferreira2019iteration} ensure that the returned triple satisfies both the inclusion constraint and the relative error criterion imposed by the outer algorithm. The algorithm is stated next.

\begin{algorithm}[H]
\caption{Inner routine for Step 1: \;\textsc{ApproxTriple} $(q_k, v_k,  \varepsilon_k)$}
\label{Alg:InnerSolver}
\begin{footnotesize}
\begin{description}
  \item[Input:] Current outer iterate $p_k\in\mathbb M$,  penalty parameter $\mu_k>0$,
        relative–error factor $0<\sigma\le 1$, maximum inner iterations
        $T_{\max}$.

  \item[Step 0.]
        Set $z_{0}\gets p_k$, iteration counter $\ell\gets0$.

  \item[Step 1.]
        Compute a Busemann subgradient $g_{\ell}\in\partial^{\,b}f(z_{\ell})$. Set residual $r_{\ell}:=\mu_k\log_{z_{\ell}}p_k-g_{\ell}$.

  \item[Stopping test:]
        If
        \[
          \bigl\|r_{\ell}\bigr\|
          \leq 
          \sigma\max\!\bigl\{\mu_k d(z_{\ell},p_k),\;\|g_{\ell}\|\bigr\}
          \quad\text{or}\quad
          t=T_{\max},
        \]
then  \textbf{stop} and return
        \(
          q_k:=z_{\ell},\;
          v_k:=g_{\ell},\;
          \varepsilon_k:=r_{\ell}.
        \)

  \item[Step 2.]
        Compute a step size $\alpha_\ell>0$  (e.g. exogenous step size),  set \(d_{\ell}:=g_{\ell}-\mu_k\log_{z_{\ell}}p_k\) and 
        \[
          z_{\ell+1}:= \exp_{z_{\ell}}\!\bigl(-\alpha_\ell d_{\ell}\bigr).
         \]
        
        \item[Step 3.]
        Set \({\ell}\gets {\ell +1}\) and go to \textbf{Step~1}.
\end{description}
\end{footnotesize}
\end{algorithm}

Algorithm~\ref{Alg:InnerSolver} describes a  subgradient-type routine for computing an inexact solution to the proximal subproblem at each iteration of BHPPM. The search point is updated by moving along a direction composed of a subgradient of the objective function corrected by the proximal term $\mu_k\log_{z_\ell}p_k$. A residual vector is computed at each step to check whether the current point satisfies the relative error criterion prescribed by the inexact inclusion condition of the outer algorithm.  If the residual is sufficiently small relative to the subgradient and proximity terms, or if the maximum number of inner iterations is reached, the procedure terminates and returns the triple $(q_k, v_k, \varepsilon_k)$. This guarantees that the returned point satisfies both the optimality inclusion and the relative error bound required in Step~1 of the BHPPM.  The flexibility of this routine allows it to handle both smooth and nonsmooth objective functions. The step size can be computed via Armijo-type backtracking or be set exogenously. 

\section{Conclusions} \label{sec:Conclusions}

We have proposed a novel variant of the proximal point algorithm for convex optimization on Hadamard manifolds, replacing classical Euclidean hyperplane projections with geometrically intrinsic projections onto horospheres defined via Busemann functions. By using a generalized notion of Busemann subdifferential and establishing convergence under inexact subgradient conditions, we demonstrated the robustness and effectiveness of the method for a broad class of convex functions on Hadamard manifolds. Importantly, the framework developed here is not tied to a smooth Riemannian structure. The key ingredients, Busemann functions, horospheres, and subdifferentials, depend only on the geodesic structure and on the convexity of the squared distance; consequently, the ideas and results extend naturally to metric Hadamard spaces, see \cite{Goodwin2024}. In particular, the algorithm and its analysis should carry over to CAT(0) spaces. As the classical proximal point method was extended from Euclidean spaces to Hadamard manifolds \cite{Ferreira2002} and then to general Hadamard spaces \cite{Bacak2013}, the same path of generalization can be pursued for the Busemann hybrid projection proximal point algorithm introduced here. A further challenge is to develop a cyclic proximal point scheme driven by horosphere projections and to analyze its convergence and rates in light of proximal results for Hadamard spaces \cite{Bacak2014}. These observations suggest that the geometric tools and proximal methodology introduced here may serve as a foundation for broader developments in convex optimization on non-Euclidean and non-smooth spaces. Promising directions for future research include the derivation of explicit complexity bounds under weaker regularity assumptions, the extension to structured or constrained problems, and the exploration of duality theory and primal-dual algorithms in the horospherical setting.

A natural extension would be to adapt the present hybrid projection–proximal scheme to the computation of zeros of (possibly set-valued) monotone vector fields on Hadamard manifolds. This would require an intrinsic notion of \emph{Busemann monotonicity} for fields, e.g., formulated via Busemann pairings and horospherical sublevel sets, and the development of the associated Busemann resolvent. We  also view this as a promising direction for future work.

\bibliographystyle{habbrv}
\bibliography{OptConvHyperbolicBusemann.bib}
\end{document}